\documentclass[notitlepage,reqno,11pt]{amsart}

\usepackage{amsmath}
\usepackage{amsthm,amssymb}
\usepackage{color,tikz}
\usepackage{pgfplots}
\usepgfplotslibrary{groupplots}
\usepackage{graphicx}
\usepackage{bbm}
\usepackage{comment}
\usepackage[a4paper,bindingoffset=0.5cm,left=2cm,right=2cm,top=2cm,bottom=2cm,footskip=.8cm]{geometry}
\usepackage[english]{babel}
\usepackage{array}
\usepackage{booktabs}
\usepackage{multirow}
\usepackage{tabularx}
\usepackage{rotating}
\usepackage{amsmath}
\usepackage{amssymb}
\usepackage{amsthm}
\usepackage{gensymb, textcomp}
\usepackage{euscript}
\usepackage{mathrsfs} 
\usepackage{libertine}
\usepackage[T1]{fontenc}
\usepackage{caption}
\usepackage{subcaption}

\usepackage{tikz}
\usetikzlibrary{arrows, automata,positioning,calc,shapes,decorations.pathreplacing,decorations.markings,shapes.misc,petri,topaths}
\usepackage{tkz-berge}
  \usepackage{pgfplots}
  \pgfplotsset{compat=newest}
  \usetikzlibrary{plotmarks}
  \usepackage{grffile}
\newlength\figureheight
  \newlength\figurewidth
\pgfplotsset{%
    tick label style={font=\scriptsize},
    label style={font=\footnotesize},
    legend style={font=\footnotesize},
   	every axis plot/.append style={very thick}
}

\newcommand{\E}{\mathbb{E}}
\newcommand{\PR}{\mathbb P}

\newtheorem{proposition}{Proposition}[section]

\DeclareMathOperator{\Var}{Var}

\newcommand{\ee}{\mathbb E}

\newcommand{\nn}{\mathbb N}
\newcommand{\pp}{\mathbb P}

\newcommand{\rr}{\mathbb R}

\usepackage{bm}

\newcommand{\cf}{\mathscr F}

\theoremstyle{remark}
\newtheorem{rem}{Remark}

\newcommand{\ve}{\varepsilon}
\newcommand{\vb}{\vspace{3.2mm}}
\newcommand{\s}{^\star}

\newcommand*\longbar[1]{%
  \hbox{%
    \vbox{%
      \hrule height 0.4pt % The actual bar
      \kern0.4ex%         % Distance between bar and symbol
      \hbox{%
       \kern-0.04em%      % Shortening on the left side
        \ensuremath{#1}%
        \kern-0.04em%      % Shortening on the right side
      }%
    }%
  }%
}

\renewcommand{\check}{\widetilde}
\renewcommand{\overline}{\longbar}
\renewcommand{\bar}{\overline}
\renewcommand{\hat}{\widehat}
\newcommand{\Pois}{\mathrm{Pois}}

\newcommand{\sumx}{\sum_{i=1}^{N} X_i}
\newcommand{\barxa}{\overline X_{N^\alpha}}
\newcommand{\barx}{\overline X_{N}}

\newcolumntype{C}[1]{>{\centering\arraybackslash}p{#1}}

\begin{document}

\title[Rare-event analysis of mixed Poisson random variables]{Rare-event analysis of mixed Poisson random variables,\\
and applications in staffing}

\author{Mariska Heemskerk, Julia Kuhn, Michel Mandjes}

\begin{abstract}
A common assumption when modeling queuing systems is that arrivals behave like a Poisson process with constant parameter. In practice, however, call arrivals are often observed to be significantly overdispersed. This motivates that in this paper we consider a {\it mixed} Poisson arrival process with arrival rates that are resampled every $N^{-\alpha}$ time units, where $\alpha>0$ and $N$ a scaling parameter.

In the first part of the paper we analyse the asymptotic tail distribution of this doubly stochastic arrival process. That is, for large $N$ and i.i.d.\ arrival rates $X_1,\ldots,X_N$, we focus on the evaluation of the probability that the scaled number of arrivals exceeds $Na$,
\[P_N(a):= {\mathbb P}\left({\rm Pois}\left(N \overline{X}_{N^\alpha}\right)\geqslant Na\right),\:\:\:\mbox{with}\:\:\:
\overline X_N:= \frac{1}{N}\sum_{i=1}^ NX_i.\]
The logarithmic asymptotics of $P_N(a)$ are easily obtained from previous results; we find constants $r_P$ and $\gamma$ such that $N^{-\gamma}\log P_N(a) \to -r_P$ as $N\to\infty$. 
Relying on elementary techniques, we then derive the exact asymptotics of $P_N(a)$: For $\alpha<\frac{1}{3}$ and $\alpha>3$ we identify (in closed-form) a function $\check{P}_N(a)$  such that $P_N(a)/\check{P}_N(a)$ tends to $1$ as $N\to\infty$.  For $\alpha\in[\frac{1}{3},\frac{1}{2})$ and $\alpha\in[2,3)$ we find a partial solution in terms of an asymptotic lower bound. 
For the special case that the $X_i$\,s are gamma distributed, we establish the exact asymptotics across all $\alpha>0$. In addition, we set up an asymptotically efficient importance sampling procedure that produces reliable estimates at low computational cost.

The second part of the paper considers an infinite-server queue assumed to be fed by such a mixed Poisson arrival process. Applying a scaling similar to the one in the definition of $P_N(a)$, we focus on the asymptotics of the probability that the number of clients in the system exceeds $Na$. 
The resulting approximations can be useful in the context of staffing. Our numerical experiments show that, astoundingly, the required staffing level can actually {\it decrease} when service times are more variable. 
\end{abstract}

\maketitle

\newpage

\section{Introduction}

In communications engineering it is increasingly accepted that traditional Poisson processes do not succeed
in capturing the variability that is typically observed in real call arrival processes \cite{kim2014call,whitt2007coping}. 
This led to the idea to instead use Cox processes \cite{cox1955cox} to model arrivals, i.e., Poisson processes in which the arrival rate follows some (non-negative) stochastic process. 
Perhaps the simplest choice, advocated in \cite{heemskerk2016scaling}, is to {\it resample} the arrival rate (in an i.i.d.\ manner) every $\Delta$ units of time;
during the resulting time intervals the arrival rate is assumed constant.
We denote these i.i.d.\ arrival rates by $(X_i)_{i\in\nn}$. This paper studies two settings in which such an overdispersed arrival process is featured.

\vb

1.\ {\it Number of arrivals}. 
We start by studying the tail asymptotics of the total number of arrivals in a time interval of given length.
We do so in a scaling regime that was proposed in \cite{heemskerk2016scaling}, in which the arrival rates and sampling frequency are jointly inflated as follows. 
In the first place, it is natural to assume that arrival rates are large,
as these represent the contributions of many potential clients; this can be achieved by letting these arrival rates be 
$NX_1,NX_2,\ldots$ for i.i.d.\ $(X_i)_{i\in\nn}$ and some large $N$. 
In addition, the sampling frequency is set to $N^{\alpha}$ (assumed to be integer) and hence the size of each time slot is assumed to be $\Delta = N^{-\alpha}$. 
Evidently, the larger $\alpha$, the more frequently the arrival rate is resampled.

The focus is on the probabilities $P_N(a)$ and $p_N(a)$, where
\[P_N(a):= {\mathbb P}\left({\rm Pois}\left(N \overline{X}_{N^\alpha}\right)\geqslant Na\right),\:\:\:\mbox{with}\:\:\:
\overline X_N:= \frac{1}{N}\sum_{i=1}^ NX_i,\]
and $p_N(a)$ denotes the corresponding probability that the  mixed Poisson random variable equals $Na$
(assumed to be integer).
We consider the situation that $a$ is larger than $\nu:={\mathbb E}X_i$,
which entails that the event under consideration is rare and that we are in the framework of large deviations theory. 

We would like to stress the important role that is played by the time-scale parameter $\alpha>0$.
One could image that in a rapidly changing environment, the inherent overdispersion of the arrival process hardly plays a role, whereas in a slowly changing random environment, overdispersion is expected to be more dominant.
Hence the parameter $\alpha$ can be tweaked in order to match any real-world scenario in that sense.
That is, if $\alpha$ is large, since the arrival rate is resampled relatively frequently,  it is anticipated that the mixed Poisson random variable behaves Poissonian with parameter $N\nu$. 
If on the contrary $\alpha$ is small, one would expect that detailed characteristics of the distribution of the $X_i$ matter. For $\alpha=1$ both effects play a role.
This intuition underlies nearly all results presented in this paper.

\vb

2.\ {\it Number of customers in an infinite-server queue.} In the second part of 
this paper we focus on a
cornerstone model in the design and performance evaluation of communication networks: the {\it infinite-server queue}. 
This model can be 
used to produce approximations for many-server systems.  

In our paper, the arrival process is the overdispersed process 
we introduced above, and the service times are i.i.d.\ samples from a (non-negative) distribution with
distribution function $F(\cdot)$.
The number of clients in this infinite-server queue, under the arrival process described above, is studied in \cite{heemskerk2016scaling}.  
As it turns out, one can prove the (conceivable) property that the number of
clients in the system at time $t$ (which we, for simplicity, assume to be a multiple of $\Delta)$, has a {\it mixed Poisson} distribution, 
i.e., a Poisson distribution with random parameter.  
This parameter is given by  
\[\sum_{i=1}^{t/\Delta} X_i\,\Delta \, f_i(t,\Delta),\]
where $f_i(t,\Delta)$ denotes the probability that a call arriving at a uniformly distributed epoch in the interval $[(i-1)\Delta,i\Delta)$
is still in the system at time $t$. 
Evidently, for small $\Delta$  this probability essentially behaves as $\overline F(t-i\Delta)$,
with $\overline F(\cdot):=1-F(\cdot)$ denoting the complementary distribution function.

We renormalize time such that $t\equiv 1$ (which can be done without loss of generality), and again impose the scaling along the lines of \cite{heemskerk2016scaling}: the arrival rates are $NX_i$ and the interval width $N^{-\alpha}$. Then the number of clients in the system is Poisson with random parameter 
\begin{equation}\label{scaledparameter}
\sum_{i=1}^{N^\alpha} (NX_i)\,N^{-\alpha}\, f_i(1,N^{-\alpha})=N^{1-\alpha}\sum_{i=1}^{N^\alpha}X_i\,\omega_i(N^{\alpha}),
\end{equation}
where $\omega_i(N):= f_i(1,N^{-1})\approx \overline F(1-i/N).$
A clearly relevant object of study concerns
the probability that the number of clients in the system exceeds some threshold $Na$:
\begin{equation}\label{def:Q}
Q_N(a) := {\mathbb P}\left({\rm Pois}\left( N^{1-\alpha} \sum_{i=1}^{N^\alpha} X_i \,\omega_i(N^{\alpha})\right)\geqslant Na\right);
\end{equation}
$q_N(a)$ denotes the corresponding probability that the  mixed Poisson random variable equals $Na$. 
To ensure that the event under consideration is rare,  $a$ is assumed to be larger than
\[\frac{\nu}{N^{\alpha}} \sum_{i=1}^{N^\alpha}\,\omega_i(N^{\alpha})
\approx 
\frac{\nu}{N^{\alpha}} \sum_{i=1}^{N^\alpha} \overline{F}(1-i/N^\alpha)
\approx \nu \int_0^1 \overline{F}(x){\rm d}x.\] 

A related question of practical interest concerns {\it staffing}: how many servers should be allocated to ensure a given service level for customers or jobs arriving according to a mixed Poisson process in a random environment?
Approximating the many-server model by its infinite-server counterpart, 
we approach this classical problem in queueing theory as an asymptotic dimensioning problem: we want to find the smallest $a$ such that $Q_N(a)$ (or $q_N(a)$) is below some desired (typically small) $\varepsilon$ as $N$ tends to infinity (cf.\ \cite{BMR04}).
The resulting procedure has applications in the context of call centers, cloud computing or in the design of data centers \cite{patch2016online,LMS16}.
 Related literature on (dynamic) staffing procedures in such settings is, e.g., \cite{JMMW96, Whitt99, whitt2007coping}; see also the recent review \cite{defraeye2016staffing} and the references therein. Previous work that addresses overdispersion in the arrival process includes \cite{BRZ10,HLW15,jongbloed2001managing}.
Our approach in this paper is different to earlier work in that it uses exact asymptotics to approximate the objective function that we want to minimize in the staffing problem.
As we focus on the large-deviations setting, the technique we develop is specifically useful in the regime in which the performance requirements are strict (i.e., the probability of service degradation should be kept low).

\vb

We now comment on the type of results we establish in this paper.
As is common in the literature, we first consider {\it logarithmic asymptotics}, that is, we identify 
a constant $r_Q>0$ (that depends on $a$) such that, for $\gamma:= \min\{\alpha,1\}$, 
\begin{equation}\label{LD1}\lim_{N\to\infty}
\frac{1}{N^\gamma}\log Q_N(a)=-r_Q.\end{equation}
This is easily done by using the techniques from \cite{heemskerk2016scaling}.

These logarithmic asymptotics provide useful insight into the decay of the probabilities of interest,
but it should be noted that they are inherently imprecise. 
More specifically, they suggest that one could use `naive' approximations of the form
\[P_N(a) \approx {\rm e}^{- r_PN^{\gamma}},\:\:\:Q_N(a) \approx {\rm e}^{- r_QN^{\gamma}}\]
for $N$ large.
It is important to notice, however, that (\ref{LD1}) only entails that $P_N(a) = \xi(N) \exp(- r_PN^{\gamma})$, with $\xi(\cdot)$ being subexponential in the sense that 
\[\lim_{N\to\infty}\frac{1}{N^{\gamma}}\log \xi(N)=0.\]
In other words, it does not rule out that, for instance, $\xi(N) =10^{10}$, or $N^M$ for some given $M$, or even $\exp(N^{0.99\,\gamma})$. 
This motivates the interest in {\it exact asymptotics}. Here, the objective is to identify a function $\check P_N(a)$ such that $\check P_N(a)/ P_N(a)\to 1$ as $N\to\infty$ (which we denote throughout the paper by $P_N(a)\sim \check P_N(a)$), leading to the evident approximation $P_N(a)\approx \check P_N(a)$. 
Along the same lines we would like to find the exact asymptotics $\check Q_N(a)$ for the probability $Q_N(a)$. 

\vb

The contributions and organization of our paper are as follows. 
In Sections \ref{sec:AboutP}, \ref{sec:exampleP} and \ref{sec:IS} we focus on the evaluation of the probabilities $P_N(a)$ and $p_N(a)$. 
After having introduced the notation, in Section \ref{sec:AboutP} we first briefly present the logarithmic asymptotics. 
We then use elementary techniques to derive the exact asymptotics, however, as it turns out, these only apply when the time scales of the arrival process and the resampling are sufficiently separated: we address the cases $\alpha<\frac{1}{3}$ and $\alpha>3$ (with a partial solution for $\alpha\in[\frac{1}{3},\frac{1}{2})$ and $\alpha\in[2,3)$ in terms of an asymptotic lower bound).

In Section \ref{sec:exampleP} it becomes clear why such elementary techniques do not work across all values of $\alpha$: for the important special case of the $X_i$ corresponding to i.i.d.\ gamma distributed random variables \cite{jongbloed2001managing} we find the exact asymptotics for all $\alpha>0$, and in the range $(\frac{1}{2},2)\setminus\{1\}$ these turn out to have a rather intricate shape.

Section \ref{sec:IS} focuses on rare-event simulation as a means to find an accurate approximation at relatively low computational cost: we propose an importance-sampling based technique, which we prove to be asymptotically efficient.

In Section \ref{sec:AboutQ} we shift our attention to the probabilities $Q_N(a)$ and $q_N(a)$. 
Again, logarithmic asymptotics can be found, and in addition we manage to identify the exact asymptotics for the case $\alpha=1$. By a series of numerical examples it is illustrated how the resulting approximation can be used for staffing purposes. We performed extensive experiments, and make the striking observation that increasing the variability of the service times (e.g.\ Pareto service times rather than exponential ones) often leads to less conservative staffing rules. 

%--------------------------------------------------------

\section{Asymptotics of $P_N(a)$}\label{sec:AboutP}

We start by introducing the framework that we consider throughout the paper.
In our setup we let $(X_i)_{i\in\nn}$ be a sequence of i.i.d.\ random variables distributed as a generic random variable $X$,
where $\nu:=\ee X_i$. 
Assume that the moment-generating function of $X$, denoted by $M_X(\vartheta):=\ee\left[{\rm e}^{\vartheta X}\right]$, is finite in an open set containing the origin. 
The {\it Fenchel-Legendre transform} (or convex conjugate) of the cumulant-generating function $\Lambda_X(\vartheta):=\log M_X(\vartheta)$ is defined as
\begin{equation}\label{def:J}
I_X(a):=\sup_{\vartheta\in\rr} \left\{ \vartheta a-\Lambda_X(\vartheta)\right\}.
\end{equation} 
We assume that the optimizing $\vartheta$ in (\ref{def:J}) indeed exists, and we denote it by $\vartheta_X^\star$ (thus suppressing that $\vartheta_X^\star$ actually depends on $a$). Under these conditions, it is known that the sample mean $\barx:=N^{-1}\sumx$ satisfies a large deviations principle with {\it rate function} $I_X(\cdot)$ (see, e.g.,  \cite{DZ1998}). Furthermore, a result by Bahadur and Rao \cite{bahadurrao} states that we have the following exact asymptotics for $\barx$: when $a>\nu$,
\begin{equation}
\label{BR1}
\lim_{N\to\infty} {\mathbb P}\left(\barx \geqslant a\right) {\rm e}^{N\,I_X(a) }\sqrt{N} = C_X(a).
\end{equation}
We assume that $X$ is non-lattice, in which case $C_X(\cdot)$ takes the form
\begin{equation}\label{def:K}
C_X(a)=\frac{1}{\vartheta_X^\star\sqrt{2\pi\Lambda''_X(\vartheta_X^\star)}},
\end{equation}
where $\Lambda_X''(\vartheta_X^\star)$ denotes the second derivative of $\Lambda_X(\vartheta)$ evaluated at $\vartheta_X^\star$; if $X$ is lattice, the constant $C_X(a)$ should be defined slightly differently \cite[Thm.\ 3.7.4]{DZ1998}.
There is also a local limit version of~(\ref{BR1}): with $\xi_N(\cdot)$ the density of $\sum_{i=1}^N X_i$, from \cite{petrov},
\begin{equation}
\label{BR1a}
\lim_{N\to\infty} \xi_N(Na)\, {\rm e}^{N I_X(a)} \sqrt{N} = C_X(a)I_X'(a).
\end{equation}

In our analysis the tail asymptotics of Poisson random variables play a crucial role. 
We note that the Bahadur-Rao asymptotics entail that for the probabilities
\begin{equation}\label{def:psi}
\psi_N(a\,|\,x):={\mathbb P}\left({\rm Pois}\left(N x\right)\geqslant Na\right),
\end{equation} 
it holds that
\begin{equation}
\label{BR2}
\lim_{N\to\infty} \psi_N(a\,|\,x) {\rm e}^{N\,I(a\,|\,x)} \sqrt{N} = C(a\,|\,x),\end{equation}
for $a>x$.
Here, $I(\cdot\,|\,x)$ is the rate function associated with a Poisson random variable  with parameter $x$, that is, $I(\cdot\,|\,x)$ is the Fenchel-Legendre transform of the cumulant-generating function $\Lambda(\vartheta)=x({\rm e}^\vartheta-1)$ of the Poisson random variable. 
Inserting the optimizer $\vartheta^\star=\log(a/x)$, this yields $I(a\,|\,x) =   a \log(a/x)- a +x$. 
Bearing in mind that the Poisson variable is lattice, it turns out that the function $C(a\,|\,x)$ takes the form ({\it cf.}~(\ref{def:K}))
\begin{align*}
C(a\,|\,x):= \frac{1}{1-\exp\left(\vartheta^\star\right)}\,\frac{1}{\sqrt{2\pi\Lambda''(\vartheta^\star)}}=\frac{1}{1-a/x}\,\frac{1}{\sqrt{2\pi a}}.
\end{align*}

Let us first present the logarithmic asymptotics of $P_N(a)$ (the same logarithmic asymptotics hold for $p_N(a)$). Here we merely state the results as the proof is exactly as in \cite[Section 4.1]{heemskerk2016scaling}. We distinguish between the cases $\alpha>1$ and $\alpha<1$; the former case we refer to as the {\it fast} regime as the $X_i$'s are sampled relatively frequently, whereas the latter case is the {\it slow} regime. For completeness, the logarithmic asymptotics for the intermediate case $\alpha = 1$, though standard, are included as well.

\begin{itemize}
\item[$\circ$] %For $\alpha>1$, the arrival rate is resampled relatively frequently, explaining why we can replace $\overline X_{N^\alpha}$ by $\nu$. 
In the fast regime $N^\alpha$ is substantially larger than $N$, and hence the rare event will be essentially due to $\barxa$ being close to $\nu$, and the Poisson random variable with parameter (roughly) $N\nu$ exceeding $Na$.
Accordingly, following the argumentation in  \cite{heemskerk2016scaling}, one obtains
\[\lim_{N\to\infty} \frac{1}{N}\log P_N(a) = -I(a\,|\,\nu).\]
This result entails that $P_N(a)$ decays essentially exponentially. 
\item[$\circ$] In the slow regime, assuming the support of $X_i$ is unbounded, the rare event will be a consequence of the joint effect of (i)~$\barxa$ being close to $a$, and (ii)~the Poisson variable with parameter (roughly) $Na$ attaining a typical value; the first event is rare, but the second is not. %For $\alpha<1$, assuming the support of $X_i$ is unbounded, it can be argued that $P_N(a)$ behaves as ${\mathbb P}(\overline{X}_{N^\alpha} \geqslant a),$ as a consequence of the fact that the $X_i$ are resampled relatively infrequently. 
In this regime, we thus have
\[\lim_{N\to\infty} \frac{1}{N^\alpha}\log P_N(a) = -I_X(a);\]
observe that this corresponds to subexponential decay.
\item[$\circ$] For $\alpha=1$, the random variable ${\rm Pois}\left(N \overline{X}_{N^\alpha}\right)$ can be written as the sum of $N$ i.i.d.\ contributions, each of them distributed as $Z:={\rm Pois}(X)$. 
Noting that
\[\log {\mathbb E}\exp \left(\vartheta Z\right) = \Lambda_X({\rm e}^\vartheta -1),\]
a straightforward application of Cram\'er's theorem \cite {DZ1998} yields that the decay is exponential:
\begin{equation}
\label{IZ}\lim_{N\to\infty} \frac{1}{N}\log P_N(a) = -\sup_\vartheta\left(\vartheta a - \Lambda_X({\rm e}^\vartheta -1)\right)=:I_Z(a).
\end{equation}
\end{itemize}

In the remainder of this section we show that for a range of values of $\alpha$ the exact asymptotics of $P_N(a)$ and $p_N(a)$ can be found relying on elementary probabilistic techniques. %We distinguish the cases $\alpha>1$ and $\alpha<1$; the former case we refer to as the {\it fast} regime, as the $X_i$'s are sampled relatively frequently, whereas the latter case is the {\it slow} regime. 
We focus on the fast regime in Section \ref{sec:AboutP:fast}, and on the slow regime in Section \ref{sec:AboutP:slow}. We conclude with the exact asymptotics for the intermediate case  $\alpha= 1$, which follow directly from the Bahadur-Rao result; see Section 2.3. %For completeness, we conclude the section with the exact asymptotics for the intermediate case $\alpha=1$, which follow directly from the Bahadur-Rao result.

\subsection{Fast regime}\label{sec:AboutP:fast}

In this section we assume that $\alpha>1$. We start by proving an upper bound for $P_N(a)$.
In self-evident notation,  we have
\begin{equation}\label{decomp}
P_N(a)=\int_{0}^\infty
\psi_N(a\,|\,x)\,{\mathbb P}\left(\barxa \in {\rm d}x\right),
\end{equation}
with $\psi_N(a\,|\,x)$ as defined in (\ref{def:psi}).
For any $\delta$, Eqn.\ (\ref{decomp}) is majorized by
 \begin{equation}
\label{split}
\int_{\nu-N^\delta}^{\nu+N^\delta}
\psi_N(a\,|\,x)\, {\mathbb P}\left(\barxa\in {\rm d}x\right)+ {\mathbb P}\left(\left|\,\barxa-\nu\,\right|\,\geqslant N^\delta\right);
\end{equation}
we determine an appropriate value for $\delta$ later on. The first term in  (\ref{split}) is evidently bounded from above by $\psi_N(a\,|\,\nu+N^\delta)$.
Motivated by (\ref{BR2}), 
we will show that, as $N\to\infty$, 
\begin{equation}\label{UB1}
\psi_N(a\,|\,\nu+N^\delta)  \,{\rm e}^{N\,I(a\,|\,\nu)} \sqrt{N}\to C(a\,|\,\nu),
\end{equation}
whereas the second term in (\ref{split}) turns out to be asymptotically negligible.

To verify that (\ref{UB1}) holds, note that $C(a\,|\,\nu)/C(a\,|\,\nu+N^\delta)\to 1$ when $\delta<0$, which follows by a standard continuity argument.
We therefore proceed by considering
$N\,I(a\,|\,\nu)-N\,I(a\,|\,\nu+N^\delta),$
which behaves as
\begin{eqnarray*}
\lefteqn{\hspace{-0.2cm}N \left(a\log\frac{a}{\nu}+a-\nu\right)- N\left(a\log\frac{a}{\nu+N^\delta}+a-(\nu+N^\delta)\right)}\\
&=& Na\log \left(1+\frac{N^\delta}{\nu}\right) + N^{1+\delta}=\left(\frac{a}{\nu}+1\right)N^{1+\delta}+ 
O(N^{1+2\delta})\to 0
\end{eqnarray*}
if $\delta<-1$. Thus, for such $\delta$ we have established (\ref{UB1}). 

Now consider the second term of (\ref{split}), and, more specifically,
\begin{equation}\label{split:2}
{\mathbb P}\left(\left|\,\barxa-\nu\,\right|\,\geqslant N^\delta\right)
 \,{\rm e}^{N\,I(a\,|\,\nu)} \sqrt{N},
 \end{equation}
for $N\to\infty$. 
Due to a Chernoff bound, we have
\[
\PR\left(\,\barxa \geq \nu + N^{\delta}\right) \leqslant \exp\left({- N^{\alpha}\sup_{\vartheta} \left({\vartheta}(\nu + N^{\delta}) - \log \E \,{\rm e}^{\vartheta X_i } \right)}\right)={\rm e}^{-N^{\alpha}I_X(\nu + N^{\delta})},
\]
and hence (\ref{split:2}) is majorized by
\[{\rm e}^{-N^\alpha I_X(\nu+ N^\delta)} \,{\rm e}^{N\,I(a\,|\,\nu)} \sqrt{N} + {\rm e}^{-N^\alpha I_X(\nu- N^\delta)} \,{\rm e}^{N\,I(a\,|\,\nu)} \sqrt{N}.\]
Now realize that $I_X(\nu+ N^\delta) = \frac12 I_X''(\nu) N^{2\delta}+O(N^{3\delta})$ and similarly for $I_X(\nu- N^\delta)$. Thus, the expression from the previous display vanishes when $\alpha+2\delta>1$, or, equivalently, $\delta>(1-\alpha)/2$, where $(1-\alpha)/2<0$ since $\alpha>1$.

We note that the requirements $\delta<-1$ (corresponding to the first term) and $\delta>(1-\alpha)/2$ (corresponding to the second term) are both fulfilled when $\alpha>3$. Thus, we have shown that for $\alpha>3$ an asymptotic upper bound for $P_N(a)$ is given by (\ref{UB1}).

Let us now turn to the corresponding lower bound. The probability of interest majorizes 
\[\psi_N(a\,|\,\nu-N^\delta)\,\int_{\nu-N^\delta}^{\nu+N^\delta}
 {\mathbb P}\left(\barxa \in {\rm d}x\right).\]
As above, we can check that for $\delta<-1$,
\[\psi_N(a\,|\,\nu-N^\delta)  \,{\rm e}^{N\,I(a\,|\,\nu)} \sqrt{N}\to C(a\,|\,\nu),\]
and, by the Bahadur-Rao result (\ref{BR1}),
\[\int_{\nu-N^\delta}^{\nu+N^\delta}
 {\mathbb P}\left(\barxa \in {\rm d}x\right)\sim
 1- 2\exp\left(- \frac12 I_X''(\nu)N^\alpha N^{2\delta}\right)\to 1,\]
 when $\delta>-\alpha/2.$ This can be realized when $\alpha>2$ (and is hence fulfilled when $\alpha>3$ as well).
 This proves the lower bound.

 Combining the upper and lower bounds, we thus find the following result.
 \begin{proposition}\label{thm:alphageq1} 
 For $\alpha >3$, as $N\to\infty$,
\[ P_N(a) \sim \,{\rm e}^{-N\,I(a\,|\,\nu)}\frac{ C(a\,|\,\nu)}{ \sqrt{N}}.\]
For $\alpha \in (2,3]$, 
\[\liminf_{N\to\infty}P_N(a) \,{\rm e}^{N\,I(a\,|\,\nu)} \sqrt{N}\geqslant C(a\,|\,\nu).\]
\end{proposition}

\begin{rem}\label{REM0}
This result is in accordance with the intuition we gave at the beginning of the section -- in the fast regime the asymptotics
of $P_N(a)$ depend on the distribution of the $X_i$ only through their mean $\nu$. This also gives an indication as to why the asymptotics for $\alpha$ closer to $1$ may be more delicate to deal with. One can imagine that for more moderate values of $\alpha$ the result may not be precise enough, and that also large deviations coming from $\overline X_{N^\alpha}$ may play a role in that regime. This is confirmed in Section \ref{sec:exampleP}, where we consider an example with $X_i\sim \mbox{Exp}(\lambda)$. It turns out that the exact asymptotic expression for $\alpha\in(1,2)$ is indeed more intricate than the expression provided in Thm.\ \ref{thm:alphageq1}.\hfill$\diamondsuit$
\end{rem}

\begin{rem} \label{REM1}
Along the same lines the asymptotics for $p_N(a)$ can be found. They turn out to be, for $\alpha>3$, as $N\to\infty$,
\[ p_N(a) \sim {\rm e}^{-N\,I(a\,|\,\nu)}\frac{C(a\,|\,\nu)}{\sqrt{N}}\left(1-{\rm e}^{-I'(a\,|\,\nu)}\right).\]
This is in line with the result of Prop.\ \ref{thm:alphageq1}: informally,
\begin{eqnarray*}p_N(a) &=& P_N(a)-P_N(a+1/N)\\
&\approx& \frac{C(a\,|\,\nu)}{\sqrt{N}}{\rm e}^{-N\,I(a\,|\,\nu)} - \frac{C(a+1/N\,|\,\nu)}{\sqrt{N}}{\rm e}^{-N\,I(a+1/N\,|\,\nu)}\\&\approx& \frac{C(a\,|\,\nu)}{\sqrt{N}} {\rm e}^{-N\,I(a\,|\,\nu)}\left(1-{\rm e}^{-I'(a\,|\,\nu)}\right),
\end{eqnarray*}
for large $N$, based on elementary Taylor arguments.\hfill$\diamondsuit$
\end{rem}

\subsection{Slow regime}\label{sec:AboutP:slow}

We now consider the slow regime, i.e., $\alpha<1.$ We have to distinguish between two cases.
\begin{itemize}
\item[$\circ$]
In Case I we assume that  $X_i$ may have outcomes larger than $a$  with positive probability:
\[b_+:=\sup\{b: {\mathbb P}(X_i>b)>0\}>a;\]
as a consequence $I_X(a)<\infty.$ Recall that in this case, for $N^\alpha$ substantially smaller than $N$, it can be argued that $P_N(a)$ essentially behaves as ${\mathbb P}(\overline{X}_{N^\alpha} \geqslant a)$.
%The intuition is now that for $N^\alpha$ substantially smaller than $N$ the rare event will be a consequence of the joint effect of (i)~$\barxa$ being close to $a$, and (ii)~the Poisson variable with parameter (roughly) $Na$ attaining a typical value; the first event is rare, but the second is not. 
\item[$\circ$] In Case II we consider the opposite situation:
$b_+<a.$
Then the intuition is that the rare event under consideration is the consequence of large deviations of both random components: of (i)~$\barxa$ being close to $b_+$, and (ii)~the Poisson variable with parameter (roughly) $Nb_+$ attaining the atypical value $Na$. 
\end{itemize}

{\it Case I}. We start by establishing an upper bound. Note that $P_N(a)$ is majorized by
\[ {\mathbb P}\left(\barxa \geqslant a-N^\delta\right)
+ \psi_N(a\,|\,a-N^\delta).\]
Due to the Bahadur-Rao result stated in (\ref{BR1}), the first term is asymptotically equivalent to
\[N^{-\alpha/2} C_X(a-N^\delta) {\rm e}^{-N^\alpha I_X(a-N^\delta)},\]
which behaves as $N^{-\alpha/2} C_X(a) {\rm e}^{-N^\alpha I_X(a)}$ when $\delta<-\alpha$ (as a direct consequence of  the standard expansion
$I_X(a-N^\delta) = I_X(a)- N^\delta I_X'(a) + O(N^{2\delta})$). 
In addition, again using the Chernoff bound, we have
\begin{equation}\label{UB2}
{\rm e}^{N^\alpha I_X(a)} \,\psi_N(a\,|\,a-N^\delta) \leqslant {\rm e}^{N^\alpha I_X(a)} \exp\left(-N\left(a\log\frac{a}{a-N^\delta}+N^\delta\right)\right).\end{equation}
Observe that the exponent in the second factor of the right hand side of (\ref{UB2}) behaves as $N^{2\delta+1}$. We conclude that (\ref{UB2}) vanishes if $2\delta+1>\alpha$, or, equivalently, $\delta>(\alpha-1)/2$ (note that $(\alpha-1)/2<0$). In order to simultaneously meet $\delta<-\alpha$ and $\delta>(\alpha-1)/2$, we need to have $\alpha<\frac{1}{3}.$

We now turn to the lower bound. The probability of interest is bounded from below by
\[\psi_N(a\,|\,a+N^\delta)\,{\mathbb P}\left(\barxa \ge a+N^\delta\right).\]
The first factor is bounded from below by $1$ minus a term that decays as $\exp(-N^{1+2\delta})$ (which goes to $1$ when $\delta>-\frac{1}{2}$),
whereas the second behaves as $N^{-\alpha/2} C_X(a) {\rm e}^{-N^\alpha I_X(a)}$ when $\delta<-\alpha$. In other words, there is an appropriate $\delta$  for all $\alpha<\frac{1}{2}.$
We have thus arrived at the following result.

 \begin{proposition}\label{thm:alphaleq1case1} 
 Assume $b_+>a.$ For $\alpha <\frac{1}{3}$, as $N\to\infty$,
\[ P_N(a) \sim {\rm e}^{-N^\alpha\,I_X(a)}  \frac{ C_X(a)}{{N}^{\alpha/2}}.\]
For $\alpha \in[\frac{1}{3},\frac{1}{2})$,  
\[\liminf_{N\to\infty}P_N(a) \,{\rm e}^{N^\alpha\,I_X(a)}  {N}^{\alpha/2}\geqslant C_X(a).\]
\end{proposition}

\begin{rem}
Note that here, in contrast with Prop.~\ref{thm:alphageq1}, the rate function is that of $X$ rather than the Poisson random variable. As expected, when $\alpha$ is small, the rare event is typically a result of a large deviation of $\overline X_{N^\alpha}$. However, for values of $\alpha$ closer to $1$ the same reasoning as in Remark \ref{REM0} applies, and we do not expect a simple asymptotic expression as given in Prop.\
\ref{thm:alphaleq1case1} to hold for all $\alpha\in(\frac 1 3,1)$ (as will be confirmed in Section~\ref{sec:exampleP},
which covers the special case in which the $X_i$ are exponentially distributed). \hfill$\diamondsuit$
\end{rem}

\begin{rem}
As in Remark \ref{REM1}, the asymptotics for $p_N(a)$ can be found as well. As it turns out, as $N\to\infty$,
\[p_N(a) \sim {\rm e}^{-N^\alpha\,I_X(a)}  \frac{C_X(a)I_X'(a)}{{N}^{1-\alpha/2}}.\]
This is consistent with the result stated in Prop.\ \ref{thm:alphaleq1case1}: 
\begin{eqnarray*}p_N(a) &=& P_N(a)-P_N(a+1/N)\\
&\approx& \frac{C_X(a)}{{N}^{\alpha/2}}{\rm e}^{-N^\alpha\,I_X(a)} - \frac{C_X(a+1/N)}{{N}^{\alpha/2}}{\rm e}^{-N^\alpha\,I_X(a+1/N)}\\&\approx& \frac{C_X(a)}{{N}^{\alpha/2}}{\rm e}^{-N^\alpha\,I_X(a)}\left(1-{\rm e}^{-N^{\alpha-1}I_X'(a)}\right)
\approx C_X(a) I_X'(a) {\rm e}^{-N^\alpha\,I_X(a)} N^{\alpha/2-1} ,
\end{eqnarray*}
for large $N$. Note that the asymptotic expansion of $P_N(a)$ has a polynomial factor $N^{-\alpha/2}$, whereas
$p_N(a)$ has a polynomial factor $N^{\alpha/2-1}$. So in this case $P_N(a)$ and $p_N(a)$ are not (asymptotically) off by a constant, but by a constant multiplied by $N^{\alpha-1}$. \hfill$\diamondsuit$
\end{rem}

{\it Case II}. %\footnote{\tt J: For this case we also have that effects of both random components play a role. If we want to cover the case of more moderate $\alpha$, then probably the route should be most similar to this case? MM: yes, one would think so. A tricky thing, though, could be that in this case we can reduce our interval to essentially point $b$, which I believe then would not hold. This probably complicates everything quite a bit.}  
In the above arguments for the slow regime, it is crucial that we assumed that $X_i$ can exceed $a$ with positive probability (which entails that $I_X(a)<\infty$). We now consider the situation that $b_+<a.$
We derive the exact asymptotics of $P_N(a)$ by separately considering a lower bound and an upper bound. 
We throughout assume that both $I_X(b_+)$ and $I_X'(b_+)$ are finite. The proof essentially follows   that of \cite{blom2016}, in which exact asymptotics of the Markov-modulated infinite-server queue are addressed.

We start with the lower bound.  Let $K$ the smallest value in $\{2,3,\ldots\}$ such that 
$-1/K$ is strictly larger than $\alpha-1$. Fix $\delta\in(\alpha-1,-1/K)$. 
%\footnote{\tt J: This should be $\delta\in(\alpha-1,-1/K)$ I think so that we have $1+\delta K<0$ --- the relationship on p.~9 that the reviewer asked about. All other bounds should still be fine since $-1/K<-\alpha/K$.}.
We have, for $\alpha<1$,
\[P_N(a) \geqslant \int_{  b_+-N^{\delta}}^{b_+}
\psi_N(a\,|\,x)\,{\mathbb P}\left(\barxa\in {\rm d}x\right)=
\int_{  b_+-N^{\delta}}^{b_+}
\psi_N(a\,|\,x)\,N^\alpha \xi_{N^\alpha}(N^\alpha x){\rm d}x ,\]
recalling that $\xi_N(\cdot)$ denotes the density of $\sum_{i=1}^N X_i$.
Fix an arbitrary $\zeta>0$. The right-hand side of the previous display majorizes, by Petrov's local limit version of the Bahadur-Rao result (\ref{BR1}), in combination with (\ref{BR2}), for $N$ sufficiently large,
\[(1-\zeta)\int_{b_+-N^{ \delta}}^{b_+}\frac{C(a\,|\,x)}{\sqrt{N}} {\rm e}^{-N\,I(a\,|\,x)}\cdot{ C_X(x)I_X'(x)}{N^{\alpha/2}} {\rm e}^{-N^\alpha I_X(x)}{\rm d}x.\]
This is in turn asymptotically equal to, with $\gamma(a) :=   C(a\,|\,b_+)\,C_X(b_+)I_X'( b_+)$, using the transformation $y:=b_+-x$,
\begin{equation}
\label{INT} (1-\zeta)\,\gamma(a) \,N^{(\alpha-1)/2}{\rm e}^{-N\,I(a\,|\,b_+)} {\rm e}^{-N^\alpha I_X(b_+)}
\int_{0}^{N^\delta} {\rm e}^{N\phi_1(y)} \,{\rm e}^{N^\alpha\phi_2(y)} 
{\rm d}y,\end{equation}
where
\[
 \phi_1(y):=-I(a\,|\,b_+-y)+I(a\,|\,b_+)=a\log\left(1-\frac{y}{b_+}\right)+y,\:\:\: \:\:\:\phi_2(y):=-I_X(b_+-y)+I_X(b_+).\]
 For all $y\in[0,N^\delta]$, there exist $\ell_i$ and $u_i$ ($i=1,2$) such that
 \[\ell_1 N^{1+K\delta} +\sum_{k=1}^{K-1}\beta_{1,k}  N\,y^k \leqslant N\phi_1(y) \leqslant 
u_1 N^{1+K\delta} +\sum_{k=1}^{K-1}\beta_{1,k}  N\,y^k,\:\:\:\:\beta_{1,1}:=1-\frac{a}{b_+},\]
  \[\ell_2 N^{\alpha+K\delta} +\sum_{k=1}^{K-1}\beta_{2,k}  N^\alpha y^k \leqslant N^\alpha\phi_2(y) \leqslant 
u_2 N^{\alpha+K\delta} +\sum_{k=1}^{K-1}\beta_{2,k}  N^\alpha y^k;\]
observe that $\beta_{1,1}<0$. 
We now further analyze the integral in (\ref{INT}). 
We find, using the above inequalities and the fact that both $1+K\delta<0$ and $\alpha+K\delta<0$ (as we have chosen $\delta< -1/K$),
\begin{align*}
\int_{0}^{N^\delta} {\rm e}^{N\phi_1(y)} \,{\rm e}^{N^\alpha\phi_2(y)} 
{\rm d}y &\geqslant  {\rm e}^{\ell_1 N^{1+K\delta}+\ell_2 N^{\alpha+K\delta}} \int_{0}^{N^\delta}
\exp\left(\sum_{k=1}^{K-1}\beta_{1,k}  N\,y^k+\sum_{k=1}^{K-1}\beta_{2,k}  N^\alpha y^k\right)
{\rm d}y\\&\sim \int_{0}^{N^\delta}
\exp\left(\sum_{k=1}^{K-1}\beta_{1,k}  N\,y^k+\sum_{k=1}^{K-1}\beta_{2,k}  N^\alpha y^k\right)
{\rm d}y.
\end{align*}
Applying the transformation $z:=N\,y$, and using that $\delta>-1$, this integral can be evaluated as
\[\frac{1}{N} 
\int_{0}^{N^{\delta+1}}
\exp\left(\sum_{k=1}^{K-1}\beta_{1,k}  N^{1-k}\,z^k+\sum_{k=1}^{K-1}\beta_{2,k}  N^{\alpha-k} z^k\right)
{\rm d}z\sim\frac{1}{N}\int_0^\infty 
{\rm e}^{\beta_{1,1}z}{\rm d}z
=
\frac{1}{N}\cdot\frac{b_+}{{a-b_+}}.
\]
Letting $\zeta\downarrow 0$, we have thus found
\[\liminf_{N\to\infty} P_N(a) N^{(\alpha+1)/2} {\rm e}^{N\,I(a\,|\,b_+)} {\rm e}^{N^\alpha I_X(b_+)} \geqslant
\gamma(a) \cdot\frac{b_+}{{a-b_+}}.\]

We proceed by the upper bound. Evidently,
\[P_N(a) = \int_{  b_+-N^{\delta}}^{b_+}
\psi_N(a\,|\,x)\,{\mathbb P}\left(\barxa\in {\rm d}x\right)
+\int_0^{  b_+-N^{\delta}}
\psi_N(a\,|\,x)\,{\mathbb P}\left(\barxa\in {\rm d}x\right).\]
The first integral in the previous display can be dealt with as in the upper bound ({\it mutatis mutandis}; e.g.\
the factor $1-\zeta$ becomes $1+\zeta$, and the $u_i$ need to be used rather than the $\ell_i$). We therefore focus on the second integral, which is clearly bounded above by 
$\psi_N(a\,|\,b_+-N^\delta)$. 
Now observe that
\[I(a\,|\,b_+)-I(a\,|\,b_+-N^\delta)=a\log\left(\frac{b_+-N^\delta}{b_+}\right)+N^\delta\leqslant
\left(1-\frac{a}{b_+}\right)N^\delta = \beta_{1,1}\,N^\delta.\]
As a consequence, as $N\to\infty$, recalling that $\delta>\alpha-1$ and $\beta_{1,1}<0$,
\[{\rm e}^{N\,I(a\,|\,b_+)} {\rm e}^{N^\alpha I_X(b_+)} \psi_N(a\,|\,b_+-N^\delta) \leqslant {\rm e}^{\beta_{1,1}N^{\delta+1}} {\rm e}^{N^\alpha I_X(b_+)}\to 0.\]
We conclude the interval $[0,b_+-N^\delta)$ does not contribute to the asymptotics. We thus have established the upper bound, leading to the following result.

\begin{proposition} 
Assume $\alpha<1$ and $b_+<a$. Then
\[\lim_{N\to\infty} P_N(a) \sim  {\rm e}^{-N\,I(a\,|\,b_+)} {\rm e}^{-N^\alpha I_X(b_+)}N^{-(\alpha+1)/2}
\gamma(a) \,\frac{b_+}{{a-b_+}},\]
where $\gamma(a) :=   C(a\,|\,b_+)\,C_X(b_+)I_X'( b_+)$.
\end{proposition}

\begin{rem}\label{REM3}
As before, we can identify the asymptotics of $p_N(a)$ as well:
\[p_N(a) \sim  {\rm e}^{-N\,I(a\,|\,b_+)} {\rm e}^{-N^\alpha I_X(b_+)}N^{-(\alpha+1)/2}\cdot
\gamma(a) \cdot\frac{b_+}{{a-b_+}}\left(1-{\rm e}^{-I'(a\,\mid\,b_+)}\right)\]
as $N\to\infty$.\hfill$\diamondsuit$
\end{rem}

\subsection{Intermediate range}\label{sec:BR} 

We finally consider the case $\alpha=1$. 
The random variable ${\rm Pois} (N \overline{X}_{N^\alpha})$ is distributed as the sum of $N$ i.i.d.\ contributions, each of them distributed as $Z:={\rm Pois}(X)$. Assuming that maximum in the definition (\ref{IZ})  of $I_Z(a)$  is attained at $\vartheta_Z^\star$, the Bahadur-Rao result yields, as $N\to \infty$,
\begin{align*}
P_N(a)\sim {\rm e}^{-N I_Z(a)}\frac{ C_Z(a)}{ \sqrt{N}},
\end{align*}
where now
\begin{eqnarray*}
C_Z(a)&:=& \frac{1}{1-{\rm e}^{\vartheta_Z^\star}} \frac{1}{\sqrt{2\pi \Lambda_Z''(\vartheta_Z^\star)}}
= \frac{1}{1-{\rm e}^{\vartheta_Z^\star}} \frac{1}{\sqrt{2\pi\left({\rm e}^{\vartheta_Z^\star} \Lambda_X'\left({\rm e}^{\vartheta_Z^\star}-1\right)+{\rm e}^{2\vartheta_Z^\star} \Lambda_X''\left({\rm e}^{\vartheta_Z^\star}-1\right)\right)}}
\\&=&\frac{1}{1-{\rm e}^{\vartheta_Z^\star}} \frac{1}{\sqrt{2\pi\left(a+{\rm e}^{2\vartheta_Z^\star} \Lambda_X''\left({\rm e}^{\vartheta_Z^\star}-1\right)\right)}}.
\end{eqnarray*} 
Based on the same arguments as in Remark \ref{REM1} we infer that
\begin{align*}
p_N(a)\sim \frac{1}{\sqrt{2\pi N\left(a+{\rm e}^{2\vartheta_Z^\star} \Lambda_X''\left({\rm e}^{\vartheta_Z^\star}-1\right)\right)}} \,{\rm e}^{-N I_Z(a)}\,.
\end{align*}

\section{Asymptotics of $P_N(a)$: special case of gamma $X_i$'s}\label{sec:exampleP}

In this section we consider the special case that the $X_i$\,s are i.i.d.\ samples from the gamma distribution. The use of this specific mixed Poisson distribution for call center staffing purposes is advocated in e.g.\ \cite{jongbloed2001managing}.
In the analysis, this can be reduced to the case where the $X_i$\,s are exponentially distributed with parameter $\lambda$ (i.e., mean $\lambda^{-1}$), see Remark \ref{JK}. 
%\footnote{MH: I rewrote it a little bittle, to avoid repetition.}
%the result for exponentially distributed $X_i$\,s can be translated into a result for gamma distributed $X_i$\,s. 

\vb

To start the exposition, we note that if the $X_i$\,s are exponential with parameter $\lambda$, then  $\sum_{i=1}^{N^{\alpha}} X_i$ has a gamma distribution with parameters
$N^{\alpha}$ and $\lambda$.
The objective of this section is to evaluate the asymptotics of $p_N(a)$ across all values of $\alpha$; later we comment on what the corresponding $P_N(a)$ looks like. 
We assume throughout that $a$ is larger than ${\lambda}^{-1}$. 
The computations are facilitated by the fact that an exact expression for $p_N(a)$ is available. 
It takes a routine calculation, which we include for completeness, to compute $p_N(a)$:
\begin{eqnarray*}
p_N(a)&=&\int_0^\infty \frac{(N^{1-\alpha} x)^{Na}}{(Na)!} \mathrm{e}^{- (\lambda + N^{1-\alpha})x} \frac{\lambda^{N^{\alpha}}}{(N^{\alpha}-1)!} x^{N^{\alpha}-1}\, \mathrm{d}x \\
&=&\frac{(N^{1-\alpha})^{Na}}{(Na)!} \frac{\lambda^{N^{\alpha}}}{(N^{\alpha}-1)!} \int_0^\infty \mathrm{e}^{- (\lambda+ N^{1-\alpha})x} x^{Na + N^{\alpha}-1}\, \mathrm{d}x \\
&=&\frac{(Na + N^{\alpha}-1)!}{(Na)!(N^{\alpha}-1)!}\frac{(N^{1-\alpha})^{Na}\lambda^{N^{\alpha}}}{(\lambda + N^{1-\alpha})^{Na+N^{\alpha}}}\int_0^\infty \frac{(\lambda + N^{1-\alpha})^{N^{\alpha}}}{(Na + N^{\alpha}-1)!} \mathrm{e}^{- (\lambda + N^{1-\alpha})x} x^{Na + N^{\alpha}-1}\, \mathrm{d}x \\
&= &{Na + N^{\alpha}-1 \choose Na} \bigg(\frac{N^{1-\alpha}}{\lambda + N^{1-\alpha}}\bigg)^{Na}\bigg(\frac{\lambda}{\lambda + N^{1-\alpha}}\bigg)^{N^{\alpha}}.
\end{eqnarray*}

\begin{rem}
We recognize here the probability that a \textit{negative binomially distributed} random variable with success probability $p:={N^{1-\alpha}}/({\lambda + N^{1-\alpha}})$ attains $Na$ successes before $N^{\alpha}$ failures have occurred. This can be understood as follows. Note that a Poisson random variable with parameter $x T$ represents the number of Exp$(x)$ ``success clocks'' expiring within a period of length $T$. In our case the rate of the success clocks is $x=N^{1-\alpha}$ and  the length of the period corresponds to the time it takes for $N$ exponential ``failure clocks'' of rate $\lambda$ to expire, that is, we have $T=\sum_{i=1}^{N^{\alpha}} X_i$.
%In our case the rate of the success clocks is $x=N^{1-\alpha}$ and the length of the period is $T=\sum_{i=1}^{N^{\alpha}} X_i$, which corresponds to the time it takes for $N^{\alpha}$ exponential ``failure clocks'' of rate $\lambda$ to expire.  
Thus, $p_N(a)$ is the probability that $Na$ success clocks expire before the $N^\alpha$th failure clock expires and the period ends. The success probability is indeed given by $p$ as it is the probability that the next Exp$(N^{1-\alpha})$ success clock expires before a Exp$(\lambda)$ failure clock.
\hfill$\diamondsuit$
\end{rem}

\iffalse
\begin{rem}
We consider the case of exponentially distributed $X_i$'s (with parameter $\lambda$), but it is clear that it also covers $X_i$'s having a gamma distribution (with parameters $\beta$ and $\lambda$). To see this, note that in the latter case  $\sum_{i=1}^{N^{\alpha}} X_i$ has a gamma distribution with parameters $N^\alpha \beta$ and $\lambda$.
\footnote{As opposed to the exponential distribution, gamma seems to be a good candidate for the distribution of a random arrival rate.}
\end{rem}
\fi

\begin{rem}\label{JK}
In the above setup we considered exponentially distributed $X_i$\,s. Note, however, that our analysis only relies on $\sum_{i=1}^{N^{\alpha}} X_i$ having a gamma distribution, and thus can easily be extended to the practically relevant case \cite{jongbloed2001managing} that the $X_i$\,s are i.i.d.\ samples from a gamma distribution. It is noted that the gamma distribution has two parameters (as opposed to the exponential distribution), and therefore allows for more modelling flexibility (e.g., the mean and variance can be fitted).
%\footnote{\tt\tiny J: I shortened this a bit (previous text is commented out). That's because I now saw that Michel's remark in the Section about $Q$ gets back to this, so I thought it's ok not to motivate the gamma distribution here in more detail. But tell me if you disagree. Before it said in addition: With gamma random variables, the spread of probability mass around the mean of a random arrival rate can be mimicked more accurately than with exponential random variables which have maximal probability mass at $0$.   Besides being more realistic, using a gamma distribution for the random arrival rates allows for more modelling flexibility compared to the exponential distribution.} 
%However, this can easily be extended to the case where the $X_i$ are gamma distributed (with parameters $\beta$ and $\lambda$), as our analysis only relies on $\sum_{i=1}^{N^{\alpha}} X_i$ having a gamma distribution. 
%This results in $\sum_{i=1}^{N^{\alpha}} X_i$ having a Gamma distribution with parameters $N^\alpha \beta$ and $\lambda$. 
%With gamma random variables, the spread of probability mass around the mean of a random arrival rate can be mimicked more accurately than with exponential random variables which have maximal probability mass at $0$.   Besides being more realistic, using a gamma distribution for the random arrival rates allows for more modelling flexibility compared to the exponential distribution.
\hfill$\diamondsuit$\end{rem}

%We anticipate on the result deviating from what we found earlier, as we work with $p_N(a)$ here rather than $P_N(a)$.

As a first step in deriving the exact asymptotics of $p_N(a)$, we approximate the binomial coefficients by applying Stirling's formula, which says that
$
n! \sim \sqrt{2 \pi n}\, n^n {\rm e}^{-n} .
$
As a consequence we find that
\begin{align*}
{ Na + N^{\alpha} - 1 \choose Na } \sim 
\frac{1}{\sqrt{2 \pi}}\frac{\sqrt{Na + N^{\alpha}-1}}{\sqrt{Na}\sqrt{N^{\alpha}-1}}
\frac{(Na + N^{\alpha}-1)^{Na + N^{\alpha}-1}}{(Na)^{Na} (N^{\alpha}-1)^{N^{\alpha}-1}}
\end{align*}
Applying this in the expression for $p_N(a)$ then yields
\begin{align}\nonumber
p_N(a)&= {Na + N^{\alpha}-1 \choose Na}\bigg(\frac{N^{1-\alpha}}{\lambda + N^{1-\alpha}}\bigg)^{Na} \bigg(\frac{\lambda}{\lambda + N^{1-\alpha}}\bigg)^{N^{\alpha}}\\
&\sim \frac{1}{\sqrt{2 \pi}}\frac{\sqrt{Na + N^{\alpha}-1}}{\sqrt{Na}\sqrt{N^{\alpha}-1}}
\frac{(Na + N^{\alpha}-1)^{Na + N^{\alpha}-1}}{(Na)^{Na} (N^{\alpha}-1)^{N^{\alpha}-1}}
\bigg(\frac{N^{1-\alpha}}{\lambda + N^{1-\alpha}}\bigg)^{Na}\bigg(\frac{\lambda}{\lambda + N^{1-\alpha}}\bigg)^{N^{\alpha}}\nonumber\\
%&=\frac{1}{\sqrt{2 \pi}}\frac{\sqrt{Na + N^{\alpha}-1}}{\sqrt{N^{\alpha}-1}\sqrt{Na}}\frac{N^{\alpha}-1}{Na + N^{\alpha} -1}\\
%&\bigg(\frac{Na + N^{\alpha}-1}{Na}\bigg)^{Na}\bigg(\frac{N^{1-\alpha}}{\lambda + N^{1-\alpha}}\bigg)^{Na}\\
%&\bigg(\frac{Na + N^{\alpha}-1}{N^{\alpha}-1}\bigg)^{N^{\alpha}}\bigg(\frac{\lambda}{\lambda + N^{1-\alpha}}\bigg)^{N^{\alpha}}\\
&=\frac{1}{\sqrt{2 \pi}} 
{\color{black}\frac{\sqrt{N^{\alpha}-1}}{\sqrt{Na}\sqrt{Na + N^{\alpha} -1}}        }               \cdot 
{\color{black}\bigg(\frac{Na + N^{\alpha}-1}{a\lambda (N^{\alpha} + \frac{N}{\lambda})}\bigg)^{Na}      }             \cdot 
{\color{black}\bigg(\frac{\lambda(Na + N^{\alpha}-1)}{(\lambda + N^{1-\alpha})(N^{\alpha}-1)}\bigg)^{N^{\alpha}}}\label{three}.
\end{align}

In order to determine the asymptotic behavior of this expression for large $N$, we again consider the three  regimes separately. We do so by evaluating the three factors in (\ref{three}). 

\subsection{Fast regime}\label{sec:exampleP:fast}

We start by examining the case $\alpha>1$. For the first factor we have
\begin{align*}
\frac{1}{\sqrt{2 \pi}}\frac{\sqrt{N^{\alpha}-1}}{\sqrt{Na}\sqrt{Na + N^{\alpha} -1}} &\sim 
\frac{1}{\sqrt{2 \pi}}\,\frac{1}{\sqrt{Na}}\,.
\end{align*}
The middle factor can be addressed as follows. 
For ease we analyze its logarithm:
\begin{eqnarray}\nonumber
Na
\log\left(\frac{Na + N^{\alpha}-1}{a\lambda (N^{\alpha} + {N}/{\lambda})}\right)%=-Na\log(a\lambda)+Na\log\left(\frac{1+N^{1-\alpha} a -N^{-\alpha}}{1+N^{1-\alpha}/\lambda}\right),
&=&-Na\log(a\lambda)\,+\\
&&Na\log\left(1+N^{1-\alpha} a -N^{-\alpha}\right) -Na\log\left(1+ N^{1-\alpha}/\lambda\right)
\label{expon}\end{eqnarray}
%or equivalently \[-Na\log(a\lambda)+Na\log\left(1+N^{1-\alpha} a -N^{-\alpha}\right) -Na\log\left(1+ N^{1-\alpha}/\lambda\right).\]
For the last factor we similarly obtain
\begin{eqnarray}\nonumber
N^\alpha\log \left(\frac{\lambda(Na + N^{\alpha}-1)}{(\lambda + N^{1-\alpha})(N^{\alpha}-1)}\right)
%&\qquad\qquad=
%N^\alpha\log \left(\frac{1+N^{1-\alpha}a-N^{-\alpha}}{1+N^{1-\alpha}/\lambda-N^{1-2\alpha}/\lambda-
%N^{-\alpha}}\right)\\
&=&N^\alpha\log (1+N^{1-\alpha}a-N^{-\alpha})-\\
&&N^\alpha\log\left(1+\frac{1}{\lambda} N^{1-\alpha}-\frac{1}{\lambda}N^{1-2\alpha}-N^{-\alpha}\right)
\label{expon2}
\end{eqnarray}

Define $\overline k:= (\alpha-1)^{-1}$ and $k_+:= {\lfloor \overline k\rfloor}$.
The logarithms can be expanded relying on their standard Taylor series form, but it can be argued that the resulting infinite series can be truncated. 
For instance,
\[Na\log\left(1+N^{1-\alpha} a -N^{-\alpha}\right) =Na \sum_{k=1}^\infty \frac{(-1)^{k+1}}{k}(N^{(1-\alpha)}a - N^{-\alpha})^k
\sim Na \sum_{k=1}^{k_+} \frac{(-1)^{k+1}a^k}{k}N^{(1-\alpha)k}.\]
Likewise, 
\[Na\log\left(1+N^{1-\alpha}/\lambda\right)\sim 
Na \sum_{k=1}^{k_+} \frac{(-1)^{k+1}\lambda^{-k}}{k}N^{(1-\alpha)k}.\]

We thus find that (\ref{expon}) asymptotically equals 
\[-Na\log(a\lambda) + Na  \sum_{k=1}^{k_+}\frac{(-1)^{k}(\lambda^{-k}-a^k)}{k}N^{(1-\alpha)k}.\]

For the last factor, note that from $k_++1$ on all terms vanish, leaving us with 
\begin{eqnarray*}N^\alpha\log(1+N^{1-\alpha}a-N^{-\alpha}) &\sim&  N^\alpha\sum_{k=1}^{k_++1} \frac{(-1)^{k+1}a^k}{k}N^{(1-\alpha)k}-1,\\
N^\alpha\log(1+N^{1-\alpha}/\lambda-N^{1-2\alpha}/\lambda-
N^{-\alpha}) &\sim& N^\alpha\sum_{k=1}^{k_++1} \frac{(-1)^{k+1}\lambda^{-k}}{k}N^{(1-\alpha)k}-1.\end{eqnarray*}%  minus 1 for N^{-alpha} term
After a bit of rewriting, we conclude that (\ref{expon2}) equals
\[N\sum_{k=0}^{k_+} \frac{(-1)^k(a^{k+1}-\lambda^{-(k+1)})}{k+1}N^{(1-\alpha)k} .\]
Defining
\[
\xi_0:= -a\log(\lambda a)+a-\frac{1}{\lambda},
\:\:\:\:
\xi_k:=(-1)^k\left(\lambda^{-k}\left(\frac{a}{k}-\frac{1/\lambda}{k+1}\right)-a^{k+1}\left(\frac{1}{k}-\frac{1}{k+1}\right)\right),
\]
we conclude that in case $\alpha>1$,
\[p_N(a)\sim \frac{1}{\sqrt{2\pi a N}} \,{\rm e}^{\xi_0 N} \exp\left(\sum_{k=1}^{k_+}\xi_k N^{(1-\alpha)k+1}\right).\]
In particular, if $\alpha>2$, then the last factor equals $1$ (the empty sum being defined as $0$). 
It is not hard to check that this result agrees with what has been found for $\alpha>3$ in Section  \ref{sec:AboutP}.%\footnote{\tt J: Since we included Remark \ref{REM1}, shouldn't we also include the reasoning in the opposite direction once?}

\subsection{Slow regime}\label{sec:exampleP:slow}

If $\alpha<1$, then the first factor behaves as
\begin{align*}
\frac{1}{\sqrt{2 \pi}}\frac{\sqrt{N^{\alpha}-1}}{\sqrt{Na}\sqrt{Na + N^{\alpha} -1}}\sim\frac{1}{\sqrt{2 \pi}}\,\frac{1}{a} N^{\alpha/2-1}\,.
\end{align*}
For the logarithm of the middle factor we now have
\begin{equation}
\label{expon3}
Na\log\left(\frac{Na + N^{\alpha}-1}{a\lambda (N^{\alpha} + {N}/{\lambda})}\right)\\
=Na\log\left(1+\frac 1 a (N^{\alpha-1} - N^{-1})\right) - Na\log\left(1+ \lambda N^{\alpha-1}\right).
%=Na\log\left(\frac{1 + \frac{1}{a}N^{\alpha-1}- \frac{1}{a}N^{-1}}{1+ \lambda N^{\alpha-1}}\right),
%Na\log\left(\frac{1+\frac 1 a (N^{1-\alpha} - N^{-\alpha})}{1+ N^{\alpha-1}}\right) \sim 
%Na \sum_{k=1}^{\lfloor \overline k\rfloor} \frac{(-1)^{k+1}\lambda^{-k}}{k}N^{(1-\alpha)k}.
\end{equation}
%or equivalently
%\[
%Na\log\left(1+\frac 1 a (N^{\alpha-1} - N^{-1})\right) - Na\log\left(1+ \lambda N^{\alpha-1}\right).
%\]
With $\widetilde{k}:=\alpha(1-\alpha)^{-1}$ and $k_-:=\lfloor \widetilde k\rfloor $, we obtain that this factor asymptotically equals 
\[
Na\sum_{k=1}^{k_-+ 1} \frac{(-1)^{k+1}(a^{-k}-{\lambda}^{k})}{k}N^{(\alpha-1)k}-1
= Na\sum_{k=0}^{k_-} \frac{(-1)^{k}(a^{-(k+1)}-{\lambda}^{k+1})}{k+1}N^{(\alpha-1)(k+1)}-1\,.
\]
For the last factor we find
\begin{eqnarray*}
N^\alpha\log \left(\frac{\lambda(Na + N^{\alpha}-1)}{(\lambda + N^{1-\alpha})(N^{\alpha}-1)}\right)
%&=N^\alpha\log \left(\frac{\lambda(a + N^{\alpha-1}-N^{-1})}{(\lambda N^{-1} + N^{-\alpha})(N^{\alpha}-1)}\right)\\
%=&N^\alpha\log \left(\frac{\lambda a(1 + \frac1a N^{\alpha-1}- \frac1a N^{-1})}{1+ \lambda N^{\alpha-1}-\lambda N^{-1} - N^{-\alpha}}\right)\\ \nonumber
&=&N^\alpha\log \left(\lambda a \right)
+ N^\alpha\log \left(1 + \frac1a N^{\alpha-1}- \frac1a N^{-1})\right)\nonumber 
\\
&&
- \,N^\alpha \log \left(1+ \lambda N^{\alpha-1}-\lambda N^{-1} - N^{-\alpha}\right) \nonumber
%&=N^\alpha\log \left(\frac{1+N^{1-\alpha}a-N^{-\alpha}}{1+N^{1-\alpha}/\lambda-N^{1-2\alpha}/\lambda-
%N^{-\alpha}}\right).\\
%&=N^\alpha\log \left(1+N^{1-\alpha}a-N^{-\alpha}\right)-N^\alpha\log \left(1+N^{1-\alpha}/\lambda-N^{1-2\alpha}/\lambda-
%N^{-\alpha}\right)\,, \nonumber
\end{eqnarray*}
where
\begin{eqnarray*}
N^{\alpha}\log\left(1+\frac 1 a N^{\alpha - 1} - \frac 1 a N^{-1}\right) 
&\sim& N^{\alpha} \sum_{k=1}^{k_-} \frac{(-1)^{k+1}}{k}a^{-k}N^{(\alpha-1)k}\,,\\
N^\alpha \log \left(1+ \lambda N^{\alpha-1}-\lambda N^{-1} - N^{-\alpha}\right) &\sim& N^{\alpha} \sum_{k=1}^{k_-} \frac{(-1)^{k+1}}{k} \lambda^k N^{(\alpha-1)k}-1\,.
\end{eqnarray*}

Combining the above we conclude
\begin{align}\label{pNalphaleq1}
p_N(a)\sim\frac{1}{\sqrt{2\pi}a} N^{\frac{\alpha}{2}-1} \,{\rm e}^{\zeta_0 N^\alpha} \exp\left(\sum_{k=1}^{k_-} \zeta_k N^{(\alpha-1)k+\alpha} \right)\,,
\end{align}
where
\begin{align*}
\zeta_0:=\log(\lambda a)+1-\lambda a\,,\quad \zeta_k:=
%(-1)^k\left(a^{-(k+1)}\left(\frac{1}{k+1}-\frac{a}{k}\right)+\lambda^{k}\left(\frac{1}{k}-\frac{\lambda}{k+1}\right)\right)\,.\\
%\zeta_k := (-1)^{k}(\frac{a^{-k}-a\lambda^{k+1}}{k+1} 
%-\frac{a^{-k}-\lambda^k}{k})\\
%(-1)^{k}\left(a^{-k}\left(\frac{1}{k+1}-\frac{1}{k}\right)+\lambda^{k}\left(\frac{1}{k}-\frac{a\lambda}{k+1}\right)\right).
 (-1)^{k}\left(\lambda^{k}\left(\frac{1}{k}-\frac{a\lambda}{k+1}\right)-a^{-k}\left(\frac{1}{k}-\frac{1}{k+1}\right)\right).
\end{align*}
It can again be verified that this result coincides for $\alpha<\frac{1}{3}$ with the one derived in Section \ref{sec:AboutP}.

\subsection{Intermediate regime}\label{sec:exampleP:intermediate}

For completeness, we also include the result for the case $\alpha=1$. %\footnote{\tt MM: I agree with Julia's comment that we better leave the full derivation out here: it repeats earlier arguments, and in addition there is an alternative approach through BR...}
We find 
\begin{align}\label{pintermediatecase}
p_N(a) \sim \frac{1}{\sqrt{2 \pi}}\,\frac{1}{\sqrt{Na(a+1)}} \exp\left(-N \left(a \log \left(a \, \frac{1+\lambda}{1+a}\right) + \log\left(\frac1\lambda \, \frac{1+\lambda}{1+a}\right)\right)\right)\,.
\end{align}
It is noted that the asymptotics of $P_N(a)$ and $p_N(a)$ could have been found by applying the Bahadur-Rao result directly, as noted in Section \ref{sec:BR}:%\footnote{\tt J: Replace $\vartheta^*$ by $\vartheta_Z^*$ as in Section 2.3} 
\[
P_N(a)\sim \frac{1}{1-{\rm e}^{\vartheta^\star_Z}} \frac{1}{\sqrt{2\pi N \Lambda_Z''(\vartheta_Z^\star)}} {\rm e}^{-NI_Z(a)}=  \frac{1}{1-a\,\frac{1+\lambda}{1+a}} \frac{1}{\sqrt{2\pi N a(a+1)}} {\rm e}^{-NI_Z(a)}\,.\]
and 
\[p_N(a)=P_N(a)-P_N\left(a+\frac 1 N\right)
\sim \frac{1}{\sqrt{2\pi N a(a+1)}} {\rm e}^{-NI_Z(a)}\,,
\]
where it can be verified that $I_Z(a)$ coincides with the exponent found in (\ref{pintermediatecase}).

\subsection{Example}

In Fig.\ \ref{fig:appro} we illustrate the accuracy of the approximation, by displaying the ratio of the approximation $\check p_N(a)$ and the exact expression for $p_N(a)$. We observe that this ratio tends to 1 as $N$ grows, as expected. 

Note that the naive approximation $p_N(a)\approx \exp\big(-N^\alpha I_X(a)\big)$ obtained from the logarithmic asymptotics is still very far off the true value for the small values $N$ considered in this example. For example, with $\alpha<1$ the ratio is as high as $\sqrt{2\pi} a N^{1-\alpha/2}\check p_N(a)/p_N(a)$ (cf.~(\ref{pNalphaleq1})). 
This shows how important it can be to consider exact asymptotics instead of logarithmic asymptotics.

\begin{figure}[h]
\centering
\begin{subfigure}{.49\linewidth}\centering
\definecolor{mycolor1}{rgb}{0.00000,0.44700,0.74100}%
\begin{tikzpicture}

\begin{axis}[%
width=0.951\figurewidth,
height=\figureheight,
at={(0\figurewidth,0\figureheight)},
scale only axis,
xmin=0,
xmax=40,
xlabel={$N$},
ymin=1,
ymax=1.03,
ylabel={$\check p_N(a)\big/p_N(a)$},
axis background/.style={fill=white}
]
\addplot [color=mycolor1,solid,forget plot]
  table[row sep=crcr]{%
3	1.02757276993117\\
4	1.02013440119929\\
5	1.01613424920412\\
6	1.01351941495354\\
7	1.01164757582831\\
8	1.01023382129067\\
9	1.00912616528668\\
10	1.00823428861604\\
11	1.00750058342135\\
12	1.00688639144432\\
13	1.00636474149692\\
14	1.00591622083755\\
15	1.00552649232912\\
16	1.00518473486185\\
17	1.00488262641993\\
18	1.00461366289632\\
19	1.00437268654823\\
20	1.004155555598\\
21	1.00395890518803\\
22	1.00377997387878\\
23	1.00361647580869\\
24	1.00346650111362\\
25	1.00332844251422\\
26	1.0032009347597\\
27	1.00308281696675\\
28	1.00297308919262\\
29	1.00287089016989\\
30	1.00277546889636\\
31	1.00268618134797\\
32	1.00260244579366\\
33	1.0025237701022\\
34	1.00244970074693\\
35	1.00237984877164\\
36	1.002313867378\\
37	1.00225144232929\\
38	1.0021922886427\\
39	1.00213616184791\\
40	1.00208283143672\\
};
\end{axis}
\end{tikzpicture}%
\caption{Fast regime, $\alpha=5$.}
\end{subfigure}
\hfill
\begin{subfigure}{.49\linewidth}\centering
\definecolor{mycolor1}{rgb}{0.00000,0.44700,0.74100}%
\begin{tikzpicture}

\begin{axis}[%
width=0.951\figurewidth,
height=\figureheight,
at={(0\figurewidth,0\figureheight)},
scale only axis,
xmin=0,
xmax=160,
xlabel={$N$},
ymin=0.91,
ymax=1,
ylabel={$\check p_N(a)\big/p_N(a)$},
axis background/.style={fill=white}
]
\addplot [color=mycolor1,solid,forget plot]
  table[row sep=crcr]{%
%5	0.911883104959935\\
6	0.910444243817053\\
7	0.910740122905081\\
8	0.911927546238935\\
9	0.913577855195394\\
10	0.915459493548991\\
11	0.917442153016514\\
12	0.919450628747533\\
13	0.921441053554711\\
14	0.923387989055136\\
15	0.925277108313745\\
16	0.927100902074308\\
17	0.92885608736885\\
18	0.930542007600124\\
19	0.932159626779592\\
20	0.933710888528919\\
21	0.935198303640022\\
22	0.936624683328871\\
23	0.937992966692615\\
24	0.939306109778961\\
25	0.940567015307685\\
26	0.941778489375343\\
27	0.942943216124293\\
28	0.944063744364714\\
29	0.945142482111692\\
30	0.946181696307577\\
31	0.947183515876428\\
32	0.948149936850648\\
33	0.949082828712751\\
34	0.949983941372631\\
35	0.950854912390761\\
36	0.951697274188627\\
37	0.952512461079747\\
38	0.953301816016589\\
39	0.954066596992007\\
40	0.954807983064431\\
41	0.955527079994453\\
42	0.956224925495566\\
43	0.956902494110019\\
44	0.95756070172478\\
45	0.958200409749\\
46	0.958822428971895\\
47	0.95942752312429\\
48	0.96001641216418\\
49	0.960589775307451\\
50	0.961148253823366\\
51	0.961692453613783\\
52	0.96222294759258\\
53	0.962740277883461\\
54	0.963244957848135\\
55	0.963737473962484\\
56	0.964218287549173\\
57	0.96468783638274\\
58	0.965146536174104\\
59	0.965594781946421\\
60	0.966032949310907\\
61	0.966461395649948\\
62	0.966880461215869\\
63	0.967290470152116\\
64	0.967691731442112\\
65	0.968084539792704\\
66	0.968469176456345\\
67	0.968845909997898\\
68	0.969214997009377\\
69	0.969576682776834\\
70	0.969931201902934\\
71	0.970278778890304\\
72	0.970619628684721\\
73	0.970953957185364\\
74	0.971281961722144\\
75	0.971603831502526\\
76	0.971919748031484\\
77	0.972229885504599\\
78	0.972534411177913\\
79	0.972833485715273\\
80	0.973127263514826\\
81	0.973415893015907\\
82	0.973699516988597\\
83	0.973978272806087\\
84	0.974252292701228\\
85	0.974521704008712\\
86	0.974786629393621\\
87	0.975047187066778\\
88	0.975303490988889\\
89	0.975555651063037\\
90	0.975803773316515\\
91	0.97604796007339\\
92	0.97628831011763\\
93	0.97652491884724\\
94	0.976757878421041\\
95	0.976987277897311\\
96	0.97721320336529\\
97	0.97743573807052\\
98	0.977654962532735\\
99	0.977870954659286\\
100	0.978083789851784\\
101	0.978293541107695\\
102	0.978500279118093\\
103	0.978704072358208\\
104	0.978904987176799\\
105	0.979103087878458\\
106	0.979298436803613\\
107	0.979491094404735\\
108	0.97968111931763\\
109	0.979868568431847\\
110	0.980053496954695\\
111	0.980235958476274\\
112	0.980416005027392\\
113	0.980593687138167\\
114	0.98076905389243\\
115	0.980942152979783\\
116	0.981113030746623\\
117	0.981281732242977\\
118	0.981448301268835\\
119	0.981612780418219\\
120	0.981775211120834\\
121	0.981935633682484\\
122	0.98209408732433\\
123	0.982250610218631\\
124	0.982405239524951\\
125	0.982558011424321\\
126	0.98270896115141\\
127	0.982858123026437\\
128	0.983005530484663\\
129	0.983151216105572\\
130	0.983295211640378\\
131	0.983437548039341\\
132	0.983578255476361\\
133	0.98371736337472\\
134	0.98385490043\\
135	0.983990894632633\\
136	0.984125373291081\\
137	0.984258363051864\\
138	0.984389889920081\\
139	0.984519979278713\\
140	0.984648655908698\\
141	0.984775944005211\\
142	0.984901867196611\\
143	0.985026448560884\\
144	0.985149710641544\\
145	0.985271675463701\\
146	0.985392364549148\\
147	0.985511798931208\\
148	0.985629999167506\\
149	0.985746985355143\\
150	0.985862777143154\\
151	0.985977393745145\\
152	0.986090853951613\\
153	0.98620317614149\\
154	0.986314378294253\\
155	0.986424477999958\\
156	0.986533492470637\\
157	0.986641438550646\\
158	0.986748332726201\\
159	0.986854191135165\\
160	0.986959029576433\\
};
\end{axis}
\end{tikzpicture}%
\caption{Slow regime, $\alpha=\frac{1}{5}$.}
\end{subfigure}
%\caption{Logarithmic importance sampling (IS) and crude Monte Carlo (MC) estimators for $P_N(a)%$, where $X_i\sim\mbox{Exp}(1)$ and $a=2$, averaged over $n= 10^7$ samples. The upper %bounds of the sample confidence intervals are indicated by dashed lines.}
\caption{Ratio of approximation $\check p_N(a)$  and exact value $p_N(a)$, where $X_i$ is exponentially distributed with parameter ${\lambda}=2.5$ and $a=1$.}
\label{fig:appro}
\end{figure}
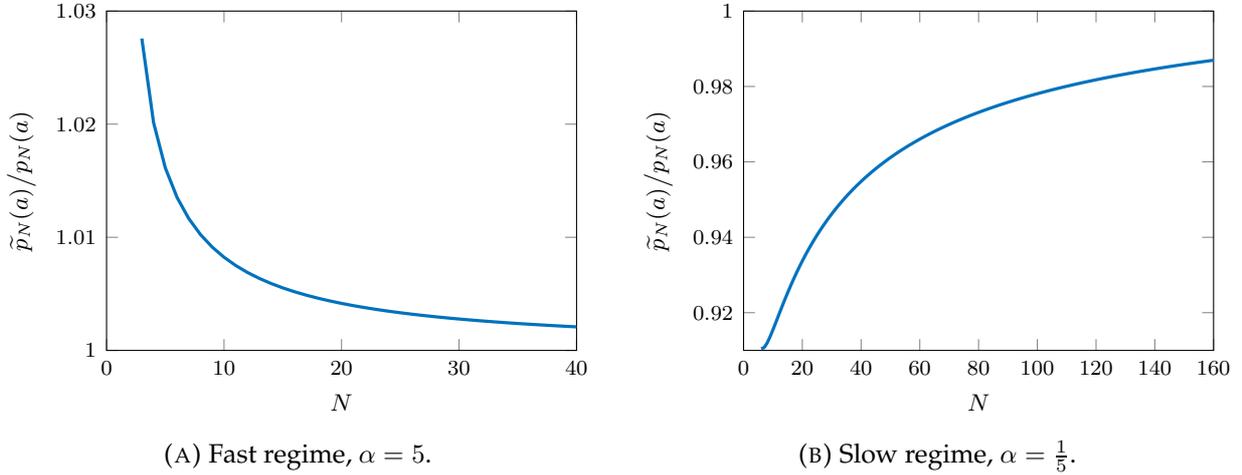

\section{Importance Sampling for $P_N(a)$}\label{sec:IS}

In the previous sections we found exact asymptotics for the rare-event probabilities $p_N(a)$ and $P_N(a)$ for (i)~a specific range of $\alpha$, and (ii)~for the specific case that the $X_i$ are exponentially distributed. To facilitate numerical evaluation (which we need, for example, if (i) and (ii) do not apply), we propose 
in this section importance sampling estimators for  $p_N(a)$ and $P_N(a)$. We establish  asymptotic efficiency properties, thus guaranteeing fast computation even for large $N$. As before, we distinguish the cases $\alpha<1$ and $\alpha>1$; the case $\alpha=1$ can be addressed by using a classical importance sampling procedure.

\subsection{Fast regime}

Recall that in this regime a rare event is typically the result of a large deviation of the Poisson random variable, while the sample mean $X_1,\dots,X_{N^\alpha}$ will typically be close to $\nu$ (under their true distribution, which we shall indicate by a subscript $\nu$). In view of this, we propose a somewhat unconventional importance sampling estimator ({\it cf.}~the more classical estimator (\ref{ISestslow}) that we will come across in the slow regime). Based on $n\in{\mathbb N}$ runs, $P_N^{(n)}(a)$ can be unbiasedly estimated by
\begin{align}\label{ISestfast}
\hat P_N^{(n)}(a)=\frac{1}{n} \sum_{i=1}^n \frac{\pp\big(\Pois\big(N \overline X_{N^\alpha,i}\big)=Z_i\big)}{\pp\left(\Pois(Na)=Z_i\right)} \,\mathbbm{1}\left\{Z_i\geq Na\right\}\,,
\end{align}
 where (i)~$Z_1,\dots,Z_n\sim \Pois(Na)$ (independently sampled), and (ii)~$\overline X_{N^\alpha,1},\ldots,\overline X_{N^\alpha,n}$ independently sampled under the original measure. 
 
Observe that the contribution from the $i$th run depends on $\overline X_{N^\alpha,i}$ as well as $Z_i$. 
It is therefore easier to analyze the corresponding estimator for $p_N(a)$,
 \[ \hat p^{(n)}_N(a):=\frac{1}{n} \sum_{i=1}^n \frac{\pp\left(\Pois\left(N \overline X_{N^\alpha,i}\right)=Z_i\right)}{\pp\left(\Pois(Na)=Z_i\right)} \,\mathbbm{1}\left\{Z_i= Na\right\}\,,\]
 which does not depend on the specific value of $Z_i$ (as it is $Na$ with certainty).
 We later comment on efficient estimation of $P_N(a)$.
 
 The contribution due to the likelihood ratio of the $i$th run
 is 
 \[ L\big(\overline X_{N^\alpha,i}\big):=\left(\frac{\overline X_{N^\alpha,i}}{a}\right)^{Na} {\rm e}^{N(a\,-\,\overline X_{N^\alpha,i})}\,.\]
The variance of the estimator (with respect to the joint distribution of $Z\sim\Pois(Na)$ and $\barxa$) can be evaluated to be
\begin{align}\label{varfast} 
\frac 1 n \,\ee\left[\left(L\big(\barxa\big)\mathbbm{1}\left\{Z= Na\right\}\right)^2\right]-p_N(a)^2
=\frac 1 n \,\ee\left[L^2\big(\barxa\big)\mathbbm{1}\left\{Z= Na\right\}\right]-p_N(a)^2\,,
\end{align}
with $Z$ distributed as each of the $Z_i$, and $\barxa$ as each of the $\overline X_{N^\alpha,i}$.
As we have seen in the introduction of Section
\ref{sec:AboutP}, the logarithmic decay rate of $p_N(a)^2$ is $-2 I(a\,|\,\nu)$.
Since the variance is always non-negative, this implies that the first term in (\ref{varfast}) vanishes no faster than with exponential rate $-2\,I(a\,|\,\nu)$. This motivates the following notion of \emph{asymptotic efficiency}
(or logarithmical efficiency), as suggested in e.g.\
\cite{sadowsky1990asymptefficient}. 

\begin{proposition}\label{thm:efficiencyfast}
The estimator $\hat p_N^{(n)}(a)$ is asymptotically efficient
 for estimating $p_N(a)$; that is
\[ \limsup_{N\to\infty} \frac 1 N \log \ee\left[L^2\left(\barxa\right)\mathbbm{1}\left\{Z= Na\right\}\right] \leqslant-2I(a\,|\,\nu)\,.\]
\end{proposition}

\begin{proof}
First, note that
\begin{align*}
\ee\left[L^2\big(\barxa\big)\mathbbm{1}\left\{Z= Na\right\}\right]
&=\ee_\nu\left[ \left(\frac{\barxa}{a}\right)^{2Na} {\rm e}^{2N (a-\barxa )} \right]\pp(Z=Na)\\
&\leqslant\ee_\nu\left[ \left(\frac{\barxa}{a}\right)^{2Na} {\rm e}^{2N (a-\barxa )} \right]\,. 
\end{align*}
Define $\cf_{\ve}^{(N)}:=\left\{\barxa\in(\nu-\ve,\nu+\ve)\right\}$, where $\ve>0$. 
Then 
\begin{align}\label{onF}
\ee_\nu\left[  \left(\frac{\barxa}{a} \right)^{2Na} {\rm e}^{2N\,(a-\barxa )} \mathbbm{1}\left\{\cf_\ve^{(N)}\right\}\right]
& \leqslant\left(\frac{\nu+\ve}{a}\right)^{2Na} {\rm e}^{2N\left(a-\nu+\ve\right)}\,.
\end{align}
On the other hand, we have
\begin{align*}
\ee_\nu\left[ \left(\frac{\barxa}{a} \right)^{2Na} {\rm e}^{2N\,(a-\barxa )}\mathbbm{1}\left\{\left(\cf_\ve^{(N)}\right)^c\right\} \right]= \ee_\nu\left[ {\rm e}^{-2NI\big(a\,\big|\,\barxa\big)}\mathbbm{1}\left\{\left(\cf_\ve^{(N)}\right)^c\right\}\right]\leqslant \pp\left(\left[\cf_\ve^{(N)}\right]^c\right)\,.
\end{align*}
where the last inequality is due to $I(a\,|\,x)\geq 0$ for any $x$. Invoking Chernoff's bound, we note that
\[ \pp\left(\left[\cf_\ve^{(N)}\right]^c\right)\leqslant2\exp\left({-N^{\alpha}j_\ve}\right),\:\:\mbox{where}\:\:
j_\ve:=\inf_{x\not
\in(\nu-\ve,\nu+\ve)}I_X(x)>0.
 \]
We conclude that for $\alpha>1$,
\[ \limsup_{N\to \infty} \frac{N^\alpha}{N} \frac{1}{N^\alpha}\log  \pp\left(\left[\cf_\ve^{(N)}\right]^c\right)\leqslant\limsup_{N\to\infty} -\frac{N^\alpha}{N} j_\ve=   -\infty\,. \]
Combining this with (\ref{onF}), we conclude that 
\begin{align*}
\limsup_{N\to\infty} \frac 1 N \log \ee\left[\left(L\left(\barxa\right)\mathbbm{1}\left\{Z= Na\right\}\right)^2\right] \leqslant2a \log\left(\frac {\nu+\ve}{ a}\right) +2(a-\nu+\ve).
\end{align*}
The desired result follows when taking $\ve\downarrow 0.$
\end{proof}

Formally, this result on asymptotic efficiency for $\hat p_N^{(n)}(a)$ does not imply asymptotic efficiency for $\hat P_N^{(n)}(a)$. In practice, however, we can use 
\[ \hat P_N^{(n)}(a)=\sum_{k=Na}^K \hat p_N^{(n)}(k/N),\]with  $K$ sufficiently large,
to estimate $P_N^{(n)}(a)$.

\subsection{Slow regime}

In the slow regime, assuming that $b_+\geq a$, the rare event is typically caused by a large deviation of $\overline X_{N^\alpha}$. 
Suppose that $\overline X_{N^\alpha,1},\ldots,\overline X_{N^\alpha,n}$ are independently sampled according to the original measure $\pp_\nu$ (where the subscript indicates that the expectation of each of the sample means $\overline X_{N^\alpha,i}$ involved is $\nu$). In this case we suggest the estimator
\begin{align}\label{ISestslow}
\hat P_N^{(n)}(a)=\frac{1}{n} \sum_{i=1}^n \frac{\pp_\nu\left(\overline X_{N^\alpha,i}\in{\rm d}Y_i\right)}{\pp_a\left(\overline X_{N^\alpha,i}\in{\rm d}Y_i\right)} \,\mathbbm{1}\left\{\Pois\left(NY_i\right)\geq Na\right\}\,,
\end{align}
where $Y_1,\dots,Y_n\sim \pp_a$. 
The measure ${\mathbb P}_a$ corresponds to the exponentially twisted version such that the mean becomes $a$ (rather than $\nu$). 

For each run we have the likelihood ratio, with $\vec{y}= (y_1,\ldots,y_{N^\alpha})$,
\[ L(\vec{y}\,)=\prod_{i=1}^{N^\alpha} M_X(\vartheta_a) \,{\rm e}^{-\vartheta_a y_i}\,,\]
where we recall that $M_X(\cdot)$ is the moment-generating function of $X$ and $\vartheta_a$ is the unique solution to 
\[ \ee_a[X]=\ee_\nu\left[X \frac{{\rm e}^{\vartheta X}}{M_X(\vartheta)}\right]= \frac{M_X'(\vartheta)}{M_X(\vartheta)} =a\,.\]
In this case we have seen before that $N^{-\alpha}\log P_N(a)\to -I_X(a)$ as $N\to\infty$.

\begin{proposition}
The estimator $\hat P_N^{(n)}(a)$ is asymptotically efficient for estimating $P_N(a)$; that is 
\[ \limsup_{N\to\infty}\frac{1}{N^\alpha} \log\ee_a\left[\left(L(\vec{X}\,) \mathbbm{1}\left\{\Pois\left(N\barxa\right)\geq Na\right\}\right)^2\right]\leqslant-2 I_X(a)\,.\]
\end{proposition}

\begin{proof}
Note that 
\begin{align*} 
\ee_a\left[\left(L(\vec{X}) \mathbbm{1}\left\{\Pois\left(N\barxa\right)\geq Na\right\}\right)^2\right]&= M(\vartheta_a)^{2N^\alpha} \,\ee_a\left[ {\rm e}^{-2\vartheta_a N^\alpha \barxa} \mathbbm{1}\left\{\Pois\left(N\barxa\right)\geq Na\right\}\right].
\end{align*}
On $\cf_\ve^{(N)}:=\left\{\barxa\in(a-\ve,\infty)\right\}$ we have \begin{align*}
\ee_a\left[ {\rm e}^{-2\vartheta_a N^\alpha \barxa} \mathbbm{1}\left\{\Pois\left(N\barxa\right)\geq Na\right\} \mathbbm{1}\left\{\cf_\ve^{(N)}\right\}\right]\leqslant {\rm e}^{-2\vartheta_a N^\alpha (a-\ve)}
%{\mathbb P}\left(\Pois\left(N(a+\ve)\right)\geq Na\right)
,
\end{align*}
while outside of $\cf_\ve^{(N)}$ we have
\begin{align*}
&\ee_a\left[ {\rm e}^{-2\vartheta_a N^\alpha \barxa}  \mathbbm{1}\left\{\Pois\left(N\barxa\right)\geq Na\right\}\mathbbm{1}\left\{\left[\cf_\ve^{(N)}\right]^c \right\}\right] \leqslant  \pp_a\left(\Pois\left(N (a-\ve)\right)\geq Na\right),
\end{align*}
where we used that $\vartheta_a>0$ because $a>\nu$ \cite[Lemma 2.2.5]{DZ1998}. 
By virtue of the Chernoff bound, 
\[ \pp_a\left(\Pois\left(N (a-\ve)\right)\geq Na\right)\leq {\rm e}^{-N I(a\,|\,a-\ve)},\:\:\mbox{where }\: I(a\,|\,a-\ve)>0.\]
This implies that
\begin{eqnarray*}
\lefteqn{\hspace{-4cm}
\limsup_{N\to\infty} \frac{1}{N^\alpha} \log \ee_a\left[ {\rm e}^{-2\vartheta_a N^\alpha \barxa}  \mathbbm{1}\left\{\Pois\left(N\barxa\right)\geq Na\right\}\mathbbm{1}\left\{\left[\cf_\ve^{(N)}\right]^c \right\}\right]}\\
&\leqslant& \limsup_{N\to\infty}-\frac{N}{N^\alpha}\, I(a\,|\,a-\ve) =-\infty.
\end{eqnarray*}

We let first $N\to\infty$ and then $\ve\downarrow 0$, to conclude that
\begin{align*}
\limsup_{N\to\infty}\frac{1}{N^\alpha} \log\ee_a\left[\left(L(\vec X) \mathbbm{1}\left\{\Pois\left(N\barxa\right)\geq Na\right\}\right)^2\right] \leqslant2\log M_X(\vartheta_a)-2\vartheta_a a =-2I_X(a)\,,
\end{align*}
as claimed.
\end{proof}

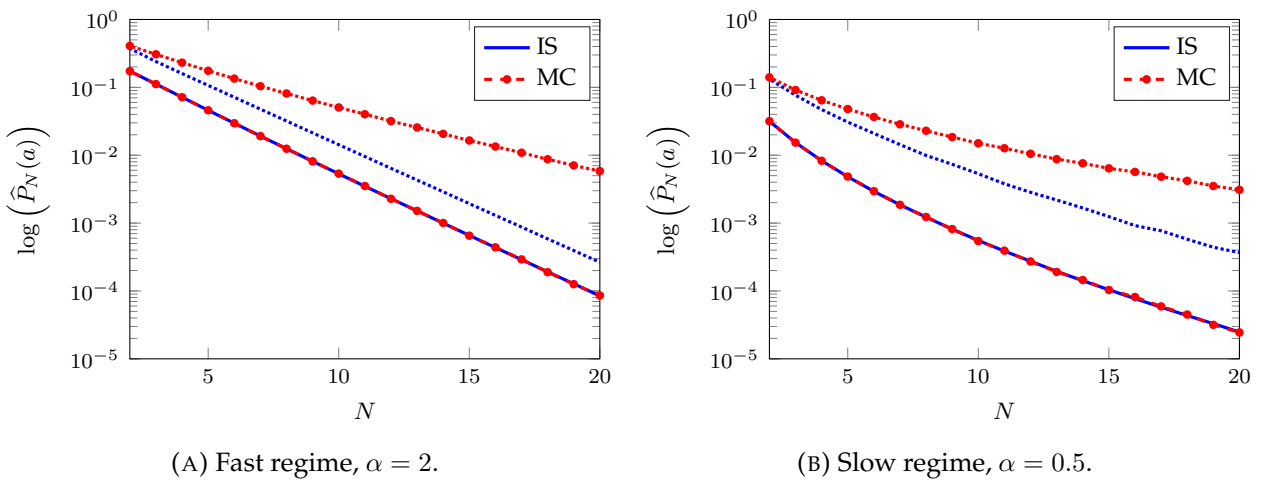
\begin{figure}[b]
\centering
\begin{subfigure}{.49\linewidth}\centering
\begin{tikzpicture}

\begin{axis}[%
width=0.951\figurewidth,
height=\figureheight,
at={(0\figurewidth,0\figureheight)},
scale only axis,
xmin=2,
xmax=20,
xlabel={$N$},
ymode=log,
ymin=1e-05,
ymax=1,
yminorticks=true,
ylabel={$\log\left(\hat P_N(a)\right)$},
axis background/.style={fill=white},
legend style={legend cell align=left,align=left,draw=white!15!black}
]
\addplot [color=blue,solid]
  table[row sep=crcr]{%
1	0.257541234337259\\
2	0.173273031129006\\
3	0.1117737663406\\
4	0.0715096806534326\\
5	0.0459026160761325\\
6	0.0296427519696038\\
7	0.0191748663853226\\
8	0.0124766618604869\\
9	0.00812831953319956\\
10	0.00532012488173537\\
11	0.00348933442988657\\
12	0.00229275952757199\\
13	0.00151275680907094\\
14	0.000998201742444984\\
15	0.00065904956424313\\
16	0.00043601526108768\\
17	0.000289145455611087\\
18	0.000191654069712327\\
19	0.00012752649366114\\
20	8.46835874059077e-05\\
};
\addlegendentry{IS};

\addplot [color=red,dashed,mark=*,mark options={scale=0.5,solid}]
  table[row sep=crcr]{%
1	0.2498301\\
2	0.1733532\\
3	0.1116705\\
4	0.0715733\\
5	0.04596\\
6	0.0295975\\
7	0.0191173\\
8	0.012472\\
9	0.0081299\\
10	0.0053547\\
11	0.0035193\\
12	0.0022823\\
13	0.0015155\\
14	0.0010044\\
15	0.000652\\
16	0.0004394\\
17	0.000292\\
18	0.0001894\\
19	0.0001259\\
20	8.59e-05\\
};
\addlegendentry{MC};

\addplot [color=blue,densely dotted,forget plot]
  table[row sep=crcr]{%
1	12.178990747914\\
2	0.374364765392218\\
3	0.240034406374641\\
4	0.159847529121203\\
5	0.106799031534158\\
6	0.071465411021558\\
7	0.0477287013570783\\
8	0.0319484820467816\\
9	0.0213468092783677\\
10	0.0142930900618132\\
11	0.00957648588273201\\
12	0.00640702779336917\\
13	0.00430613226035744\\
14	0.0028930008913035\\
15	0.00193548971039258\\
16	0.0012999004293301\\
17	0.00087429535783344\\
18	0.000585731607574214\\
19	0.000394452995372994\\
20	0.000264956141655423\\
};
\addplot [color=red,densely dotted,mark=*,mark options={scale=0.5,solid},forget plot]
  table[row sep=crcr]{%
1	0.518152495743764\\
2	0.407981150274116\\
3	0.306884617906032\\
4	0.231346664553585\\
5	0.17574798723123\\
6	0.134638158057601\\
7	0.103991624259764\\
8	0.0812578235959149\\
9	0.0637872838572574\\
10	0.0505876284617451\\
11	0.0402236972916134\\
12	0.0318583358256339\\
13	0.0256259786440347\\
14	0.0206374801923142\\
15	0.0164732504222936\\
16	0.0134288629741324\\
17	0.0108816264287226\\
18	0.00871853461643708\\
19	0.00708000054416662\\
20	0.00583012438853375\\
};
\end{axis}
\end{tikzpicture}%
\caption{Fast regime, $\alpha=2$.}
\end{subfigure}
\hfill
\begin{subfigure}{.49\linewidth}\centering
\begin{tikzpicture}

\begin{axis}[%
width=0.951\figurewidth,
height=\figureheight,
at={(0\figurewidth,0\figureheight)},
scale only axis,
xmin=2,
xmax=20,
xlabel={$N$},
ymode=log,
ymin=1e-05,
ymax=1,
yminorticks=true,
ylabel={$\log\left(\hat P_N(a)\right)$},
axis background/.style={fill=white},
legend style={legend cell align=left,align=left,draw=white!15!black}
]
\addplot [color=blue,solid]
  table[row sep=crcr]{%
1	0.0816924062626448\\
2	0.031645000237114\\
3	0.0153355543542141\\
4	0.00830347222690131\\
5	0.00482010263142992\\
6	0.00293451551612167\\
7	0.00186100159116158\\
8	0.00122313781430492\\
9	0.000813088269992746\\
10	0.000554500307864043\\
11	0.000386802637158802\\
12	0.000274159634926092\\
13	0.000195005743927094\\
14	0.000141458364520637\\
15	0.000104136947463776\\
16	7.71253662203043e-05\\
17	5.75983190172043e-05\\
18	4.34687023557782e-05\\
19	3.2999648191763e-05\\
20	2.48325235442429e-05\\
};
\addlegendentry{IS};

\addplot [color=red,dashed,mark=*,mark options={scale=0.5,solid}]
  table[row sep=crcr]{%
1	0.0815232\\
2	0.03173\\
3	0.0153232\\
4	0.0082806\\
5	0.0048504\\
6	0.0029522\\
7	0.0018601\\
8	0.0012298\\
9	0.0008158\\
10	0.0005444\\
11	0.0003925\\
12	0.0002721\\
13	0.000191\\
14	0.0001449\\
15	0.0001034\\
16	8.15e-05\\
17	5.92e-05\\
18	4.49e-05\\
19	3.17e-05\\
20	2.45e-05\\
};
\addlegendentry{MC};

\addplot [color=blue,densely dotted,forget plot]
  table[row sep=crcr]{%
1	0.262239880278689\\
2	0.132148058965068\\
3	0.0765636407968221\\
4	0.0470700661911483\\
5	0.0307756014682522\\
6	0.0209746680468316\\
7	0.0143264363191715\\
8	0.00992359421038451\\
9	0.00734549539716365\\
10	0.00534160861520568\\
11	0.00379134453948535\\
12	0.00284294742350659\\
13	0.00218766222786955\\
14	0.00166969983459211\\
15	0.00124596863196973\\
16	0.000918427490160608\\
17	0.000769960736211126\\
18	0.000578523324285712\\
19	0.000440549181260732\\
20	0.000367585479122495\\
};
\addplot [color=red,densely dotted,mark=*,mark options={scale=0.5,solid},forget plot]
  table[row sep=crcr]{%
1	0.251124787728228\\
2	0.140370111779765\\
3	0.0914569222920224\\
4	0.064447782374837\\
5	0.0479118322244084\\
6	0.0365788629944778\\
7	0.0285668686321424\\
8	0.0229519917512595\\
9	0.0185114950151745\\
10	0.0150021592029045\\
11	0.0126693623752888\\
12	0.0104947190947896\\
13	0.0087560601248878\\
14	0.00760524754374577\\
15	0.00640562110996742\\
16	0.00567671448177791\\
17	0.00482791648485345\\
18	0.00419800998488442\\
19	0.00352128161099865\\
20	0.00309234766135371\\
};
\end{axis}
\end{tikzpicture}%
\caption{Slow regime, $\alpha=0.5$.}
\end{subfigure}
%\caption{Logarithmic importance sampling (IS) and crude Monte Carlo (MC) estimators for $P_N(a)%$, where $X_i\sim\mbox{Exp}(1)$ and $a=2$, averaged over $n= 10^7$ samples. The upper %bounds of the sample confidence intervals are indicated by dashed lines.}
\caption{Logarithmic importance sampling (IS) and crude Monte Carlo (MC) estimators for $P_N(a)$, where $X_i$ is exponentially distributed with parameter ${\lambda_{\alpha}}$ (where $\lambda_{2}=1$, $\lambda_{0.5}=2.5$) and $a=2$, averaged over $n= 10^7$ samples. The upper bounds of the sample confidence intervals are indicated by dashed lines; the width of the intervals is inflated by a factor $10^3$ for better visibility.}
\label{fig:IS}
\end{figure}

\subsection{Numerical example}

We provide a numerical example with exponentially distributed $X_i$. Specifically, we consider $X_i\sim \mbox{Exp}(1)$, $a=2$, and $\alpha\in\{0.5,2\}$. Fig.~\ref{fig:IS}
shows the logarithm of $\hat P_N(a)$ as well as the corresponding crude Monte Carlo estimators, as a function of $N$. We generated $\sum_{i=1}^{N^\alpha} X_i$ by drawing from the gamma distribution with parameters $N^{\alpha}$ and $1/\lambda$. This allowed us to include values of $N$ for which $N^\alpha\notin\nn$ in Fig.~\ref{fig:IS}.({\sc b}). The dotted lines in the figures indicate the upper bounds of the standard normal $95\%$ confidence intervals evaluated using sample standard deviations (multiplied by a factor $10^3$ to make them visible). It can be seen that for the importance sampling estimator the width of the confidence interval hardly depends on $N$. In contrast, for the Monte Carlo estimator the width of the confidence interval increases significantly.
%\footnote{\tt\tiny J: I guess we can add the approximation to these figures. (We just have to choose smaller/larger $\alpha$ then -- or use the results from the exponential example. MM: take $N$ larger in right panel, to get probabilities in the order of $10^{-4}$ or so. MH: working on it, I need really specific values to get them all in line the way we want.}

\iffalse
\begin{figure}[h]
\centering
\subfloat[Fast regime, $\alpha=2$]{
 \input{pic_simP_Exp_IS_lambda1alpha2a2runs1e7.tikz}
}
\hfill
\subfloat[Slow regime, $\alpha=0.5$.]{
\input{pic_simP_Exp_IS_lambda1alpha05a2runs1e7.tikz}
}
\caption{Logarithmic importance sampling (IS) and crude Monte Carlo (MC) estimators for $P_N(a)$, where $X_i\sim\mbox{Exp}(1)$ and $a=2$, averaged over $n= 10^7$ samples. The upper bounds of the sample confidence intervals are indicated by dashed lines.}
\label{fig:IS}
\end{figure}
\fi

\section{Asymptotics for infinite-server system, and implications for staffing} \label{sec:AboutQ}

%\footnote{\tt J: In the intro we mention that $a$ should be large, if that's needed, we should repeat that here (also in the proposition). But isn't this more a requirement to ensure that we get something interesting rather than an actual condition?}
In this section we investigate the asymptotic behavior of $Q_N(a)$ ($q_N(a)$), the probability that the number of clients in the system exceeds (equals) some threshold $Na$. We consider the scaled system previously studied in \cite{heemskerk2016scaling}. 

We start by presenting the logarithmic asymptotics, which can be identified with exactly the same techniques as in \cite[Section 4.1]{heemskerk2016scaling}. 
%\footnote{\tt \tiny  "..., and assume that the latter is twice differentiable". J: Or what is needed for these log asymptotics? MM: I guess that's what it is, because we must have that specific Riemann sums converge to corresponding integrals; see below. But perhaps we should leave this technicality out here, as we (i) do not provide a proof of the logarithmic asymptotics here anyway, and hence including those details would just raise questions, and (ii) we do include these detailed aspects in the exact asymptotics.}
%Important special cases are  $\overline F(x)={\rm e}^{-\mu x}$ (i.e., exponential service times with mean $\mu^{-1}$), or $\overline F(x)=1_{\{x < m\}}$  (i.e., deterministic service times, with constant value $m$), or $\overline F(x)=(1+Hx)^{-2}$ (i.e., Pareto(2) service times, with mean $H^{-1}$).
As before, we distinguish three cases (where it is noted that $q_N(a)$ has the same logarithmic asymptotics as $Q_N(a)$); the intuition behind the three regimes is as before.
\begin{itemize}
\item[$\circ$] For $\alpha>1$,
\[\lim_{N\to\infty} \frac{1}{N}\log Q_N(a) = -I\left(a\,\left|\,\nu\int_0^1 \overline F(x){\rm d}x\right.\right),\]
where $\overline F(\cdot)$ denotes the complementary distribution function of the service times.
\item[$\circ$] For $\alpha<1$, assuming the support of $X_i$ is unbounded,
\[\lim_{N\to\infty} \frac{1}{N^\alpha}\log Q_N(a) = -\sup_{\vartheta}\left(\vartheta a - \int_0^1 \Lambda_X\left(\vartheta \overline F(x)\right){\rm d}x\right).\]
\item[$\circ$] For $\alpha=1$, \begin{equation}\lim_{N\to\infty} \frac{1}{N}\log Q_N(a) = -\sup_\vartheta\left(\vartheta a - \int_0^1 \Lambda_X\left(({\rm e}^\vartheta-1) \overline F(x)\right){\rm d}x\right).\label{a1}\end{equation}
\end{itemize}
%These expressions can be obtained with exactly the same techniques as in \cite{heemskerk2016scaling}.

In the remainder of this section, we first determine the exact asymptotics for 
the special case $\alpha=1$. That is, we assume that the arrival rates are resampled every $1/N$ time units, and we are interested in the number of customers present at time $1$ (that is, after $N$ time periods of length $1/N$). As it turns out, the case $\alpha\not = 1$ is considerably harder to deal with, and therefore left for future research. We conclude this section by a set of numerical experiments.

\subsection{Exact asymptotics}

As mentioned in the introduction ({\it viz.}\ Eqn.\ (\ref{scaledparameter})), under the scaling of \cite{heemskerk2016scaling} the number of clients in the system at time $1$ is distributed as the sum of $N$ Poisson random variables, say, $Z_1$ up to $Z_N$,  where $Z_i$ can be interpreted as the contribution due to arrivals in the interval $[(i-1)/N,i/N)$;
%\footnote{\tt\tiny J: Shouldn't these intervals also be $[(i-1)/N,i/N)$ as for $\omega_i(N)$ below?}
for details we refer to \cite{heemskerk2016scaling}.
Then it can be argued that
\[ Z_i \stackrel{\rm d}{=} {\rm Pois}\left(N X_i\cdot N^{-1}\omega_i(N)\right)=\Pois\left(X_i\,\omega_i(N)\right),\]
where we defined
$\omega_i(N)$ as the probability that a call that arrived at a uniform epoch in the interval $[(i-1)/N,i/N)$ is still present at time 1. It can be verified that
%\footnote{\tt \scriptsize As long as the $\Lambda_i$ are i.i.d.\ this is fine, but if we involve a deterministic trend it is no longer symmetric. MM: but in our case it is fine, so we can see this, right?}
\[\omega_i(N) = {N}\int_{(i-1)/N}^{i/N} \overline{F}(1-x){\rm d}x;\]
because the $X_i$ are i.i.d., we can reverse time, and hence replace $ \overline{F}(1-x)$ in the previous display by $ \overline{F}(x).$

We now wish to evaluate
\[Q_N(a)= {\mathbb P}\left({\rm Pois}\left(\sum_{i=1}^{N} X_i\, \omega_i(N)\right)\geqslant Na\right),\:\:\:
q_N(a)= {\mathbb P}\left({\rm Pois}\left(\sum_{i=1}^{N} X_i \,\omega_i(N)\right)= Na\right).\:\:\:
\]
Let $S_N=\sum_{i=1}^N Z_i$, where
$Z_i\stackrel{\rm d}{=} {\rm Pois}\left(X\, \omega_i(N)\right),$
with the $Z_i$ independent; hence $q_N(a) = {\mathbb P}(S_N= Na)$. 
It is immediately verified that, with $M_X(\cdot)$ the moment generating function of the $X_i$,
\begin{equation}\label{MGFSN}
{\mathbb E}\left[{\rm e}^{\vartheta S_N}\right]=\prod_{i=1}^N M_X\left(\omega_i(N) (\mathrm{e}^\vartheta -1)\right).
\end{equation}
Bearing in mind 
(\ref{a1}), we define
\begin{equation*}\label{def:varthetastar}\vartheta^\star :=\arg\sup_\vartheta \left\{\vartheta a - \int_0^1 \Lambda_X \left(\overline{F}(x) ({ \rm e}^\vartheta-1)\right){\rm d}x\right\}.
\end{equation*}
\iffalse
The shape of the rate function underlying this definition was proven in \cite{heemskerk2016scaling}. Note that from the G\"artner-Ellis theorem we would have expected this shape because for large $N$ the probabilities $\omega_i(N)$ are close to $\overline F(1-i/N)$, and we consider all time intervals up to time $1$.\footnote{\tt J: I added this paragraph but am not content with it yet -- I just thought some remark on the rate function assumed here was in place as it came a bit out of the blue otherwise (?)}  \fi

The idea is now that we construct a measure ${\mathbb Q}$, under which the event of interest is not rare so that a central limit theorem applies. Concretely, we choose $\mathbb{Q}$ to be an $\vartheta^\star$-twisted version of 
the original measure such that $S_N$ has moment generating function ({\it cf.}~(\ref{MGFSN}))
\begin{equation}\label{MGF:Q}{\mathbb E}_{\mathbb Q}\left[{\rm e}^{\vartheta S_N}\right]={\displaystyle \prod_{i=1}^N M_X\left(\omega_i(N) ({\rm e}^{\vartheta+\vartheta^\star} -1)\right)}\left/{\displaystyle
\prod_{i=1}^N M_X\left(\omega_i(N) ({\rm e}^{\vartheta^\star} -1)\right)}\right..
\end{equation}
As a consequence,
$q_N(a) = {\mathbb E}_{\mathbb Q} LI_N,$
with the indicator function $I_N := 1_{\{S_N= Na\}}$ and the likelihood ratio 
\[L := {\rm e}^{-\vartheta^\star S_N} 
\prod_{i=1}^N M_X\left(\omega_i(N) ({\rm e}^{\vartheta^\star} -1)\right).\]
It thus follows that
\[q_N(a) = {\rm e}^{-\vartheta^\star Na} \left(\prod_{i=1}^N M_X\left(\omega_i(N) ({\rm e}^{\vartheta^\star} -1)\right) \right){\mathbb Q}(S_N=Na).\]
 We now point out how to evaluate the middle factor in the previous display (i.e., the product), namely, we check that asymptotically this middle factor behaves as
 %\footnote{\tt J: Is the point that to evaluate this expression we don't need to compute $\omega_i(N)$ as well as a product with $N$ terms? I mean is that why we prefer expression (29)?} 
\begin{equation}\label{integ} \exp\left(N\int_0^1 \Lambda_X\big(\tau \overline F(x)\big) {\rm d}x\right), \end{equation}
with $\tau := {\rm e}^{\vartheta^\star} -1$.
\begin{comment}
\begin{equation}
\label{integ}\sum_{i=1}^N \Lambda_X\big(\tau \omega_i(N))\big) - N\int_0^1 \Lambda_X\big(\tau \overline F(x)\big) {\rm d}x
=\sum_{i=1}^N \Lambda_X\left(\tau {N}\int_{(i-1)/N}^{i/N} \overline{F}(x){\rm d}x\right) - N\int_0^1 \Lambda_X\big(\tau \overline F(x)\big) {\rm d}x.\end{equation}

To this end, we first concentrate on the argument of $\Lambda_X(\cdot)$ in the sum in the expression (\ref{integ}). We observe that
\end{comment}
The logarithm of the middle factor is 
\begin{equation*}
\sum_{i=1}^N \Lambda_X\big(\tau \omega_i(N))\big) 
=\sum_{i=1}^N \Lambda_X\left(\tau {N}\int_{(i-1)/N}^{i/N} \overline{F}(x){\rm d}x\right),
\end{equation*}
where, by a Taylor expansion of $\overline F$,
\[{N}\int_{(i-1)/N}^{i/N} \overline{F}(x){\rm d}x=\overline F\left(\frac{i-1}{N}\right)+\frac{1}{2N}\, \overline F\,'\left(\frac{i-1}{N}\right)+O\left(\frac{1}{N^2}\right).\]

As a consequence, from a Taylor expansion of $\Lambda_X(\cdot)$ we have
\[\sum_{i=1}^N \Lambda_X\big(\,\tau \omega_i(N)\big) = 
\sum_{i=1}^N \Lambda_X\left(\tau \overline F\left(\frac{i-1}{N}\right)\right) +\frac{\tau}{2N}\sum_{i=1}^N\overline F\,'\left(\frac{i-1}{N}\right)
\Lambda_X'\left(\tau \overline F\left(\frac{i-1}{N}\right)\right)+O\left(\frac 1 N\right),\]
%
\iffalse
{\bf [J: If I just plug in the Taylor expansion we got for $\omega_i(N)$, this is (I think)
\begin{align*}
&\sum_{i=1}^N \Lambda_X\big(\,\tau \omega_i(N)\big)=\\
&\qquad\sum_{i=1}^N \Lambda_X\left(\tau \overline F\left(\frac{i-1}{N}\right)\right) +\sum_{i=1}^N 
\Lambda_X'\left(\tau \overline F\left(\frac{i-1}{N}\right)\right)\left[\frac{\tau}{2N}\overline F\,'\left(\frac{i-1}{N}\right)+O\left(\frac{1}{N^2}\right)\right]+O\left(\frac{1}{N^2}\right).
\end{align*}
]}
\fi
%
%
where, as $N\to\infty$,
\begin{eqnarray*}
\frac{\tau}{2N}\sum_{i=1}^N\overline F\,'\left(\frac{i-1}{N}\right)
\Lambda_X'\left(\tau \overline F\left(\frac{i-1}{N}\right)\right) &\to &\frac{\tau}{2}\int_0^1
\overline F\,'(x) \,\Lambda_X'\left(\tau \overline F(x)\right) {\rm d}x\\&=&\frac{1}{2} \left(\Lambda_X\big(\tau \overline F(1)\big)-\Lambda_X\big(\tau \overline F(0)\big)\right),
\end{eqnarray*}
provided that $\overline F(\cdot)$ is twice differentiable on $[0,1]$ 
(recognize the {\it left Riemann sum} approximation).
%\footnote{\tt \tiny J: Someone commented out this footnote, but I still haven't understood this:  MM: this is actually the left Riemann sum at Wiki site on the Riemann sum. It means that $\overline F\,'(x) \,\Lambda_X'\left(\tau \overline F(x)\right)$ should have a bounded derivative, and therefore the second derivative of $\overline F(\cdot)$ on $[0,1]$ should be bounded. J: But that should also hold for $\Lambda_X''$ then I guess. In any case, the criterion I know is that a function is Riemann integrable on a compact interval if and only if it is bounded and continuous a.e. (see German wiki on Riemannsches Integral, subsection Riemann-Integrierbarkeit). Since we assume that $\bar F$ is twice diff'able (and $\Lambda_X$ is diff'able anyway), the integrand is continuous and thus bounded on $[0,1]$. So I don't think we need to assume anything in addition actually. However,  if we do want to assume that $\bar F'$ has bounded derivative, then we should also assume that $\Lambda_X'$ has bounded derivative as the integrand is the product of the two. So assuming that for only one of them is funny, or am I missing something?}
%\label{footnoteRI}}
%and all further terms in this expansion are $O(1/N)$.
\iffalse
Now notice that, with $G(x) = \Lambda_X(F(x)(e^{\vartheta^\star}-1))$,
\[\log  \left(  \prod_{i=1}^N M\left(F\left(\frac{i}{N}\right) (e^{\vartheta^\star} -1)\right) 
\right) = \sum_{i=1}^M G(i/N).\]\fi
Now recall the {\it trapezoidal rule} version of the Riemann sum approximation, that holds for any Riemann-integrable $G(\cdot)$:
\[\frac{1}{N}\sum_{i=1}^N G(i/N) = \int_0^1 G(x){\rm d}x +\frac{1}{2N}(G(1)-G(0)) + O\left(\frac{1}{N^2}\right).\]
Since $\Lambda_X$ is Riemann integrable on $[0,1]$, this can be applied to yield%\footnote{\tt\tiny J: Please double check that I put the right assumptions now, also in the previous paragraph.} 
\begin{eqnarray*}N\int_0^1 \Lambda_X\big(\tau \overline F(x)\big) {\rm d}x &=&\sum_{i=1}^N \Lambda_X\left(\,\tau \overline F\left(\frac{i}{N}\right)\right)-\frac{1}{2} \left(\Lambda_X\big(\tau \overline F(1)\big)-\Lambda_X\big(\tau \overline F(0)\big)\right) +O\left(\frac{1}{N}\right) \\
&=& \sum_{i=1}^N \Lambda_X\left(\,\tau \overline F\left(\frac{i-1}{N}\right)\right)+\frac{1}{2} \left(\Lambda_X\big(\tau \overline F(1)\big)-\Lambda_X\big(\tau \overline F(0)\big)\right) +O\left(\frac{1}{N}\right) .\end{eqnarray*}
We have thus arrived at
\[q_N(a) \sim  {\rm e}^{-\vartheta^\star Na}\exp\left(N\int_0^1 \Lambda_X\big(\,\overline F(x)(e^{\vartheta^\star}-1)\big) {\rm d}x \right) 
%{\frac{M\big(\overline F(1)(e^{\vartheta^\star}-1)\big)}{M\big(e^{\vartheta^\star}-1\big)}}
\,{\mathbb Q}(S_N=Na).\]
We are left to evaluate ${\mathbb Q}(S_N=Na)$. We do so by first proving the claim that, under ${\mathbb Q}$, $S_N$ obeys a central limit theorem: as $N\to\infty$, 
\[\frac{S_N-Na}{\sqrt{N}}\]
converges to a zero-mean Normal random variable. Recall from (\ref{MGF:Q}) that we have
\[\log {\mathbb E}_{\mathbb Q} \,{\rm e}^{\vartheta S_N} = \sum_{i=1}^N \Lambda_X\left(\omega_i(N)({\rm e}^{\vartheta+\vartheta^\star}-1)\right)-
\sum_{i=1}^N \Lambda_X\left(\omega_i(N)({\rm e}^{\vartheta^\star}-1)\right).\]
In order to establish that $S_N$ satisfies the anticipated central limit theorem, we prove that
$\Psi_N(\vartheta):=\log {\mathbb E}_{\mathbb Q} \,{\rm e}^{\vartheta S_N/\sqrt{N}} -\vartheta a \sqrt{N} \to \frac{1}{2}\sigma^2\vartheta^2,$
for some $\sigma^2>0.$  This is done as follows. Observe that we can write the logarithmic moment generating function $\Psi_N(\vartheta)$ as
\[\sum_{i=1}^N \Lambda_X\left(\omega_i(N)\left({\rm e}^{\vartheta^\star}-1+\left({\rm e}^{\vartheta^\star}({\rm e}^{\vartheta/\sqrt{N}}-1)\right)\right)\right)-
\sum_{i=1}^N \Lambda_X\left(\omega_i(N)({\rm e}^{\vartheta^\star}-1)\right)-\vartheta a\sqrt{N}.\]
 By applying a Taylor expansion to ${\rm e}^{\vartheta/\sqrt{N}}-1$, this can be written as (neglecting higher order terms)
\[\sum_{i=1}^N \Lambda_X\left(\omega_i(N)\left({\rm e}^{\vartheta^\star}-1+\left({\rm e}^{\vartheta^\star}\left(\frac{\vartheta}{\sqrt{N}}
+\frac{\vartheta^2}{2{N}}
\right)\right)\right)\right)-
\sum_{i=1}^N \Lambda_X\left(\omega_i(N)({\rm e}^{\vartheta^\star}-1)\right)
-\vartheta a\sqrt{N}.\]
This can be expanded to, up to terms that are $o(1)$ as $N\to\infty$,
\begin{equation}\label{expr}
\begin{aligned}
\lefteqn{\sum_{i=1}^N\left[ \Lambda_X'\left(\omega_i(N)({\rm e}^{\vartheta^\star}-1)\right)
\omega_i(N){\rm e}^{\vartheta^\star}\left(\frac{\vartheta}{\sqrt{N}}
+\frac{\vartheta^2}{2{N}}\right)\right.}\\
&&+\left.\frac{1}{2} \Lambda_X''\left(\omega_i(N)({\rm e}^{\vartheta^\star}-1)\right)
\omega_i(N)^2{\rm e}^{2\vartheta^\star}\frac{\vartheta^2}{N}\right]-\vartheta a\sqrt{N}.
\end{aligned}
\end{equation}
Now note that, similar to what we have seen before,
\[\frac{1}{N} \sum_{i=1}^N
 \Lambda_X'\left(\omega_i(N)({\rm e}^{\vartheta^\star}-1)\right)\omega_i(N)\,{\rm e}^{\vartheta^\star} = \int_0^1 \Lambda_X'\big(\,\overline F\left(x\right)({\rm e}^{\vartheta^\star}-1)\big)\overline F(x)\,{\rm e}^{\vartheta^\star}  {\rm d}x +O\left(\frac{1}{N}\right),\]
 where the integral equals $a$ by the definition of $\vartheta^\star$. 
We conclude that (\ref{expr}) converges to $ \frac{1}{2}\sigma^2\vartheta^2$ as $N\to\infty$, where the corresponding variance is given by
  \begin{eqnarray*}
  \sigma^2&:=& \int_0^1 \Lambda_X'\big(\,\overline F(x)({\rm e}^{\vartheta^\star}-1)\big)\overline F(x)\,{\rm e}^{\vartheta^\star}{\rm d}x
 + \int_0^1\Lambda_X''\big(\,\overline F(x)({\rm e}^{\vartheta^\star}-1)\big)
\overline F\,^2(x)\,{\rm e}^{2\vartheta^\star}{\rm d}x\\
&=&a
 + \int_0^1\Lambda_X''\big(\,\overline F(x)({\rm e}^{\vartheta^\star}-1)\big)
\overline F\,^2(x)\,{\rm e}^{2\vartheta^\star}{\rm d}x.
\end{eqnarray*}
We have thus established that, under ${\mathbb Q}$, $S_N$ satisfies the claimed central limit theorem. 
It directly implies that, by applying the usual continuity correction idea, ${\mathbb Q}(S_N= Na)$ behaves inversely proportionally to $\sqrt{N}$ in the sense that
\[\sqrt{N}\,{\mathbb Q}(S_N= Na) \sim \sqrt{N} \,{\mathbb P}\left({\mathscr N}(0,\sigma^2) \in\left(-\frac{1}{2\sqrt{N}},\frac{1}{2\sqrt{N}}\right)\right)
\to\frac{1}{\sqrt{2\pi}\sigma}.\]
Upon combining the above, we conclude that the following asymptotic relationship holds.

\begin{proposition}\label{thm:6.1}As $N\to\infty$, if  $\overline F(\cdot)$ is twice differentiable on $[0,1]$, 
\[q_N(a) \sim \check q_N(a) :={\rm e}^{-\vartheta^\star Na}\exp\left(N\int_0^1 \Lambda_X\big(\,\overline F(x)({\rm e}^{\vartheta^\star}-1)\big) {\rm d}x \right)
%{\frac{M\big(\,\overline F(1)(e^{\vartheta^\star}-1)\big)}
%{M\big(e^{\vartheta^\star}-1\big)}}
\,\frac{1}{\sqrt{2\pi N}\sigma}.\]
\end{proposition}

Similar to Remark \ref{REM1}, we can convert the asymptotics of $q_N(a)$ into those of $Q_N(a)$. More precisely, it can be argued that $Q_N(a)$ has the same asymptotics as $q_N(a)$, except that the expansion for $q_N(a)$ should be divided by $1-{\rm e}^{-\vartheta^\star}$ (which is smaller than 1). Note also that for the case $\overline F(\cdot)\equiv 1$ we indeed recover the expression that we provided in Section \ref{sec:BR}. Furthermore, it is easily verified that if $\pp(X_i=\lambda)=1$ (so the arrival rates are deterministic), the approximation we obtained in Prop.\ \ref{thm:6.1} coincides with that of the transient distribution of an M/G/$\infty$ queue. With $\varrho(1) :=\lambda \int_0^1 \overline F(x){\rm d}x,$ recall that the number of customers present at time $1$ is Poisson with mean $\varrho(1)$. By applying Stirling's approximation, and using that $\vartheta\s = \log(a/\varrho(1)),$
\[
q_N(a)= \left({N}\varrho(1)\right)^{Na} {\rm e}^{-N\varrho(1)}  \frac{1}{(Na)!}
\sim \left(\frac{\varrho(1) }{a}\right)^{Na} {\rm e}^{N\left(a-\varrho(1)\right)} \frac{1}{\sqrt{2\pi Na}}=\check q_N(a).
\]

\subsection{Numerical example}

We consider the following numerical example, which illustrates how Prop.\ \ref{thm:6.1} can be useful in devising staffing rules with possible applications in cloud provisioning, call center staffing or the design of data centers.
Per time slot of length $1$ time unit (which we refer to as $\Delta$) a new arrival rate is sampled from a given distribution with a mean such that on average $\lambda$ clients arrive in the time slot of length $\Delta$. 
The service times have a fixed mean $E$. 

Let us assume the system starts empty, say at 8 {\sc am}. 
Suppose we wish to determine an appropriate staffing rule for slot 100 (evidently, any other slot for which we wish to adapt staffing levels can be dealt with analogously).
Then we choose $N=100$ (recall the way we normalized time), and after scaling we have $\ee[NX_i\Delta]=\lambda$ (as $N\Delta = 1$). 
Suppose the service facility wishes to maintain a rather strict quality level; its objective is to choose the number of servers in slot 100 to be $\lfloor Na\rfloor$ (or, alternatively, $\lceil Na\rceil$), where $a$ is the smallest number such that $Q_{N}(a)$ drops below $\varepsilon$.

For the service times we consider the following three distributions: 
\begin{itemize}
\item[$\circ$]
In the first place, we assume that the service times are exponential with mean service time $E$, that is, $\overline F(x) = e^{- x/E}$.
\item[$\circ$]
A second choice is to assume that the service times are deterministically equal to $E$, that is we define $\overline F(x)=\mathbbm{1}\{x< E\}$.
\item[$\circ$]
A third choice is to assume that the service times have a Pareto(2) distribution with mean $E$, that is, $\overline F(x) = (1+x/E)^{-2}$.
\end{itemize}

As indicated in the introduction, in practice arrival rates for modeling call centers are typically not constant over time, but may be fluctuating around some mean value \cite{jongbloed2001managing}. We assume that arrival rates follow a Poisson distribution in Section \ref{sec:AboutQ:PoisX}. In Section \ref{sec:AboutQ:burstyX} we consider discrete arrival rates alternating between two values (corresponding to busy and quiet periods), motivated by applications in cloud computing, where the workload of virtual machines exhibits such bursty behaviour \cite{zhang2016burstiness}.

\subsubsection{Poisson arrival rates}\label{sec:AboutQ:PoisX}

In this example we take $X_i\sim \Pois(\lambda)$. We then have
\[\Lambda_X(\vartheta) = \lambda\left(e^\vartheta-1\right);\quad \Lambda_X'(\vartheta)=\Lambda_X''(\vartheta)=\lambda \,e^\vartheta. \]
To compute $\vartheta\s$ and $\sigma^2$, we evaluate
\[ \int_0^1 \Lambda_X\left(\,\overline F(x)(e^{\vartheta}-1)\right) {\rm d}x = \int_0^1 \lambda\,\left(\exp\left(\overline F(x)(e^{\vartheta}-1)\right)-1\right){\rm d}x
\]
and 
\[\int_0^1 \lambda\,\exp\left(\overline F(x)(e^{\vartheta^\star}-1)\right)
\overline F\,^2(x)\,{\rm e}^{2\vartheta^\star}{\rm d}x\]
by numerical integration. Inserting the resulting quantities into the formula provided in Prop.~\ref{thm:6.1}, we can compute the approximation $\check Q_N(a)$ as $\check q_N(a)(1-e^{-\vartheta\s})^{-1}$ for various $a$. Consider Fig.~\ref{fig:Q}  for a comparison of $\check Q_N(a)$ with the corresponding estimators $\hat Q_N(a)$ that are obtained by crude Monte Carlo estimation of the probability $Q_N(a)$ as defined in (\ref{def:Q}).

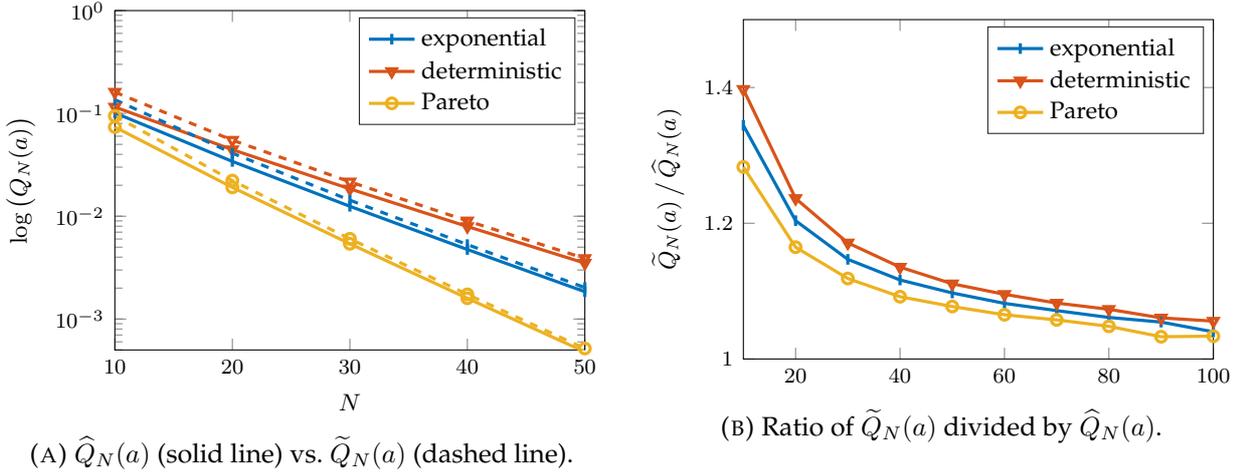
\begin{figure}[h!]
\centering
\begin{subfigure}{.49\linewidth}\centering
\definecolor{mycolor1}{rgb}{0.00000,0.44700,0.74100}%
\definecolor{mycolor2}{rgb}{0.85000,0.32500,0.09800}%
\definecolor{mycolor3}{rgb}{0.92900,0.69400,0.12500}%
\begin{tikzpicture}

\begin{axis}[%
width=0.951\figurewidth,
height=\figureheight,
at={(0\figurewidth,0\figureheight)},
scale only axis,
xmin=10,
xmax=50,
xlabel={$N$},
ymode=log,
ymin=5e-04,
ymax=1,
yminorticks=true,
ylabel={$\log\big(Q_N(a)\big)$},
axis background/.style={fill=white},
title style={font=\bfseries},
title={},
legend style={legend cell align=left,align=left,draw=white!15!black}
]
\addplot [color=mycolor1,solid,mark=|,mark options={solid}]
  table[row sep=crcr]{%
10	0.101158186\\
20	0.034111992\\
30	0.01248641\\
40	0.004744078\\
50	0.001843836\\
60	0.00072905\\
70	0.000291146\\
80	0.000117416\\
90	4.7586e-05\\
100	1.9546e-05\\
110	8.034e-06\\
120	3.238e-06\\
130	1.356e-06\\
140	5.46e-07\\
150	2.78e-07\\
};
\addlegendentry{exponential};

\addplot [color=mycolor2,solid,mark=triangle,mark options={solid,rotate=180}]
  table[row sep=crcr]{%
10	0.115338664\\
20	0.044407568\\
30	0.018449648\\
40	0.007941508\\
50	0.003498384\\
60	0.001560734\\
70	0.000704582\\
80	0.000320302\\
90	0.000147256\\
100	6.7618e-05\\
110	3.0816e-05\\
120	1.4388e-05\\
130	6.466e-06\\
140	3.14e-06\\
150	1.478e-06\\
};
\addlegendentry{deterministic};

\addplot [color=mycolor3,solid,mark=o,mark options={solid}]
  table[row sep=crcr]{%
10	0.073626586\\
20	0.019068136\\
30	0.005390578\\
40	0.001590388\\
50	0.000479348\\
60	0.000147162\\
70	4.5628e-05\\
80	1.432e-05\\
90	4.556e-06\\
100	1.436e-06\\
110	4.76e-07\\
120	1.46e-07\\
130	4.8e-08\\
140	1e-08\\
150	0\\
};
\addlegendentry{Pareto};

\addplot [color=mycolor1,dashed,forget plot,mark=|,mark options={solid}]
  table[row sep=crcr]{%
10	0.135974445326379\\
20	0.0410636663890204\\
30	0.014319488938802\\
40	0.00529631216957505\\
50	0.00202317753155426\\
60	0.000788785275981375\\
70	0.00031188955105709\\
80	0.000124600641032884\\
90	5.01717720360492e-05\\
100	2.03280687071619e-05\\
110	8.27780520419602e-06\\
120	3.38482598740596e-06\\
130	1.38889917021721e-06\\
140	5.7160208155132e-07\\
150	2.3584551121881e-07\\
};
\addplot [color=mycolor2,dashed,forget plot,mark=triangle,mark options={solid,rotate=180}]
  table[row sep=crcr]{%
10	0.161161410329637\\
20	0.0549132719252726\\
30	0.0216054318577292\\
40	0.00901622712004458\\
50	0.00388598479037055\\
60	0.00170939226036417\\
70	0.00076260491203345\\
80	0.000343743808681189\\
90	0.000156167221422136\\
100	7.13908245194872e-05\\
110	3.28002621753013e-05\\
120	1.51326258527073e-05\\
130	7.005909148923e-06\\
140	3.25314498607554e-06\\
150	1.51444347185868e-06\\
};
\addplot [color=mycolor3,dashed,forget plot,mark=o,mark options={solid}]
  table[row sep=crcr]{%
10	0.0944557560539271\\
20	0.0222086705691918\\
30	0.00602956614590235\\
40	0.00173630424424244\\
50	0.000516392530463864\\
60	0.0001567467193279\\
70	4.8254097390055e-05\\
80	1.50088478218497e-05\\
90	4.70522314113267e-06\\
100	1.48426214184722e-06\\
110	4.70569092717298e-07\\
120	1.49809118955958e-07\\
130	4.78592984372018e-08\\
140	1.53349766020519e-08\\
150	4.92618409034436e-09\\
};
\end{axis}
\end{tikzpicture}%
\caption{$\hat Q_N(a)$ (solid line) vs.\ $\check Q_N(a)$ (dashed line).}
 \end{subfigure}
 \hfill
 \begin{subfigure}{.49\linewidth}\centering
\definecolor{mycolor1}{rgb}{0.00000,0.44700,0.74100}%
\definecolor{mycolor2}{rgb}{0.85000,0.32500,0.09800}%
\definecolor{mycolor3}{rgb}{0.92900,0.69400,0.12500}%
\begin{tikzpicture}

\begin{axis}[%
width=0.951\figurewidth,
height=\figureheight,
at={(0\figurewidth,0\figureheight)},
scale only axis,
unbounded coords=jump,
xmin=10,
xmax=100,
ymin=1,
ymax=1.5,
ylabel={$\check Q_N(a)\,\big/\,\hat Q_N(a)$},
axis background/.style={fill=white},
legend style={legend cell align=left,align=left,draw=white!15!black}
]
\addplot [color=mycolor1,solid,mark=|,mark options={solid}]
  table[row sep=crcr]{%
10	1.34417639049378\\
20	1.20378975197404\\
30	1.14680592250311\\
40	1.11640495151535\\
50	1.09726544635979\\
60	1.08193577392686\\
70	1.07124793422232\\
80	1.06118962520342\\
90	1.05433892397027\\
100	1.04001170097011\\
110	1.03034667714663\\
120	1.04534465330635\\
130	1.02426192493894\\
140	1.0468902592515\\
150	0.8483651482691\\
};
\addlegendentry{exponential};

\addplot [color=mycolor2,solid,mark=triangle,mark options={solid,rotate=180}]
  table[row sep=crcr]{%
10	1.3972886865556\\
20	1.23657462901982\\
30	1.17104845890443\\
40	1.13532935055213\\
50	1.11079423824559\\
60	1.09524894079591\\
70	1.08235082933349\\
80	1.07318658229168\\
90	1.0605151669347\\
100	1.0557961566371\\
110	1.06439064691398\\
120	1.05175325637387\\
130	1.08349971372147\\
140	1.0360334350559\\
150	1.02465728813172\\
};
\addlegendentry{deterministic};

\addplot [color=mycolor3,solid,mark=o,mark options={solid}]
  table[row sep=crcr]{%
10	1.28290283694435\\
20	1.1647006592145\\
30	1.11853796492739\\
40	1.0917488337704\\
50	1.07728107859815\\
60	1.06513039594392\\
70	1.05755451455367\\
80	1.04810389817386\\
90	1.03275310384826\\
100	1.0336087338769\\
110	0.988590530918693\\
120	1.02608985586273\\
130	0.997068717441704\\
140	1.53349766020519\\
150	inf\\
};
\addlegendentry{Pareto};

\end{axis}
\end{tikzpicture}%
\caption{Ratio of $\check Q_N(a)$ divided by $\hat Q_N(a)$.}
 \end{subfigure}
 \caption{Comparison of crude Monte Carlo estimators $\hat Q_N(a)$  and the approximation $\check Q_N(a)$ as provided in Prop.\ \ref{thm:6.1}. Parameters are chosen as $a=0.2$, $\lambda=0.1$, $E=1$.}
  \label{fig:Q}
 \end{figure}

We then proceed to find the value of $a$, denoted by $a(\ve)$, for which we have $|\check Q_N(a)-\ve|<10^{-9}$ using a bisection method. The results are displayed in Table \ref{TaQ}; together with $M_1$,  the expected number of customers present at time $1$; the Monte Carlo estimates $\hat Q_N\big(a(\ve)\big)$; and the values of $\check Q_N(\underline a)$ and $\check Q_N(\overline a)$, where $\underline a$ and $\overline a$ are such that the number of servers is integer-valued: $N\underline a=\lfloor N a(\ve)\rfloor$ and $N\overline a=\lceil N a(\ve)\rceil$. Surprisingly, the results we obtain for $a(\ve)$ and $M_1$ suggest that the number of servers required decreases as the variability of the service distribution increases: a relatively small number of servers suffices when service times are Pareto(2), whereas a large number of servers is required for deterministic service times.

\begin{table}[h!]
\caption{
\label{TaQ}\small Values of $a(\ve)$ needed to achieve $|\check Q_N\big(a(\ve)\big)- \ve|<10^{-9}$ with $N=100$, expected arrival rate $\lambda=2$, and mean service time $E$.  The Monte Carlo estimates $\hat Q_N\big(a(\ve)\big)$ are also provided (based on $10^9$ runs) together with $\mbox{\sc ci}$, the width of the standard normal 95\% confidence interval, as well as the values of the approximation $\check Q_N(a)$ with $\underline{\smash{ a}}$ ($\overline a$, respectively) such that $N \underline{\smash{a}}=\lfloor N a\rfloor$ ($N \overline a=\lceil N a\rceil$, respectively). The inferred number of servers is $N\overline a$, which should be larger than the expected number of customers $M_1$ at time $1$.}
\begin{tabular}{>{\small}C{0.8cm}>{\small}C{0.8cm}>{\small}C{0.6cm}>{\small}C{1.1cm}>{\small}C{0.6cm}>{\small}C{0.6cm}>{\small}C{3.2cm}>{\small}C{3cm}}
\midrule
$F$ & $\ve$ & $E$  & $a\left(\ve\right)$ & $N\overline a$ & $\lceil M_1\rceil$ & $\frac 1 \ve \left[\hat Q_N\left(a\left(\ve\right)\right)\pm \frac{ \mbox{\sc ci}}{2}\right]$ & $\frac 1 \ve \left(\check Q_N(\underline a),\check Q_N(\overline a)\right)$\\
\midrule
\multirow{6}{*}{\rotatebox[origin=c]{90}{\parbox[c]{2.3cm}{\centering Exponential}}} 
& \multirow{3}{*}{$10^{-3}$} 
& $0.05$ & $0.2516$ & $26$ & $10$ & $0.5568\pm 0.0015$ & $(1.1009,0.6033)$\\
&& $0.5$ & $1.2602$ & $127$ & $87$  & $0.7215\pm 0.0017$ & $(1.0053,\,0.7802)$ \\
&& $1$ & $1.7537$ & $176$ & $127$ & $0.8099\pm 0.0018$ & $(1.0784,\,0.8780)$\\
\cmidrule{2-8}
& \multirow{3}{*}{$10^{-4}$} 
& $0.05$ & $0.2885$ & $29$ & $10$ & $0.8436\pm 0.0057$ & $(1.7277,0.9039)$\\
&& $0.5$ & $1.3460$ & $135$ & $87$ & $0.8380\pm 0.0057$ & $(1.1858
,0.8921)$ \\
&& $1$&  $1.8587$ & $186$ & $127$ & $0.9122\pm 0.0059$ & $(1.2238,0.9702)$ \\
\midrule
\multirow{6}{*}{\rotatebox[origin=c]{90}{\parbox[c]{2.3cm}{\centering Deterministic}}} 
& \multirow{3}{*}{$10^{-3}$} 
& $0.05$ & $0.2782$ & $28$ & $10$ & $0.8382\pm 0.0018$ & $(1.4983,0.9133)$ \\
&& $0.5$ &$1.4809$ & $149$ & $100$ & $0.7645\pm 0.0017$ & $(1.0185,\,0.8279)$  \\
&& $1$ & $2.6636$ & $267$ & $200$ & $0.8353\pm 0.0018$ & $(1.0565,\,0.9070)$\\
\cmidrule{2-8}
& \multirow{3}{*}{$10^{-4}$}
& $0.05$ & $0.3223$ & $33$ & $10$ & $0.6146\pm 0.0049$ & $(1.1319,0.6547)$  \\
&& $0.5$ & $1.5857$ & $159$ & $100$ & $0.8463\pm 0.0057$ & $(1.1407,0.9036)$  \\
&& $1$ & $2.8048$ & $281$ & $200$ & $0.8590\pm 0.0057$ & $(1.0869,0.9136)$ \\
\midrule
\multirow{6}{*}{\rotatebox[origin=c]{90}{\parbox[c]{2.3cm}{\centering Pareto(2)}}} 
& \multirow{3}{*}{$10^{-3}$} 
& $0.05$ & $0.2350$ & $24$ & $10$ & $0.6630\pm 0.0016$ & $(1.3845,0.7229)$ \\
&& $0.5$ & $1.0074$ & $101$ & $67$ & $0.8559\pm 0.0018$ & $(1.2375,\,0.9268)$\\
&& $1$ & $1.4250$ & $143$ & $100$ & $0.8224\pm 0.0018$ & $(1.1252,\,0.8894)$\\
\cmidrule{2-8}
& \multirow{3}{*}{$10^{-4}$} 
& $0.05$ & $0.2688$ & $27$ & $10$ & $0.5721\pm 0.0057$ & $(1.8616,0.9194)$\\
&& $0.5$ & $1.0818$ & $109$ & $67$ & $0.7223\pm 0.0053$ & $(1.0613,0.7633)$\\
&& $1$ & $1.5167$ & $152$ & $100$ & $0.8642\pm 0.0058$ & $(1.1959,0.9164)$\\
\midrule
\end{tabular}
\end{table}

At first sight, this outcome may seem counter-intuitive: one would perhaps have expected that unsteady service times would imply that more servers are needed. It is, however, easy to see that this conclusion is not necessarily valid (and in fact false for the example at hand). While it is true that customers arriving at an early slot can be served in time by the `deterministic servers' with probability $1$, customers arriving in later slots can never complete their service in time. For `random servers' instead, customers arriving early may not finish their service in time but on the other hand customers arriving late still have a chance of completing their service.

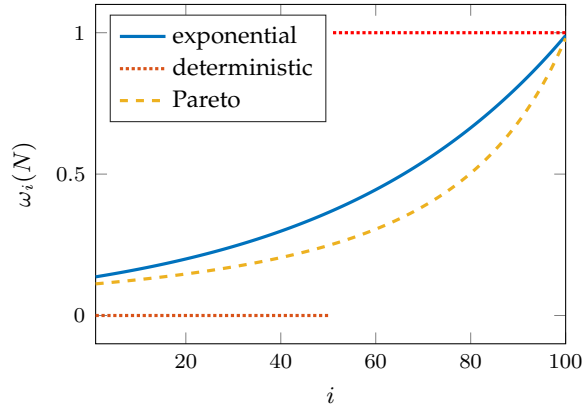
\begin{figure}[h!]
\definecolor{mycolor1}{rgb}{0.00000,0.44700,0.74100}%
\definecolor{mycolor2}{rgb}{0.85000,0.32500,0.09800}%
\definecolor{mycolor3}{rgb}{0.92900,0.69400,0.12500}%
\begin{tikzpicture}

\begin{axis}[%
width=0.951\figurewidth,
height=\figureheight,
at={(0\figurewidth,0\figureheight)},
scale only axis,
xmin=1,
xmax=100,
xlabel={$i$},
ymin=-0.1,
ymax=1.1,
ylabel={$\omega_i(N)$},
axis background/.style={fill=white},
legend style={at={(0.03,0.97)},anchor=north west,legend cell align=left,align=left,draw=white!15!black}
]
\addplot [color=mycolor1,solid]
  table[row sep=crcr]{%
1	0.136697703714006\\
2	0.139459180507609\\
3	0.142276442832896\\
4	0.145150617632361\\
5	0.148082854614246\\
6	0.151074326712441\\
7	0.154126230555672\\
8	0.157239786946169\\
9	0.160416241348003\\
10	0.163656864385288\\
11	0.166962952350448\\
12	0.170335827722751\\
13	0.173776839697317\\
14	0.177287364724817\\
15	0.180868807062071\\
16	0.184522599333765\\
17	0.188250203105528\\
18	0.192053109468568\\
19	0.195932839636139\\
20	0.199890945552036\\
21	0.203929010511402\\
22	0.208048649794063\\
23	0.212251511310661\\
24	0.21653927626184\\
25	0.220913659810753\\
26	0.225376411769145\\
27	0.229929317297304\\
28	0.234574197618146\\
29	0.239312910745734\\
30	0.244147352228501\\
31	0.249079455907501\\
32	0.254111194689968\\
33	0.259244581338505\\
34	0.264481669276218\\
35	0.269824553408111\\
36	0.275275370959076\\
37	0.280836302328812\\
38	0.286509571964013\\
39	0.292297449248178\\
40	0.298202249409394\\
41	0.304226334446456\\
42	0.310372114073703\\
43	0.316642046684929\\
44	0.323038640336779\\
45	0.329564453752003\\
46	0.336222097342979\\
47	0.343014234255914\\
48	0.349943581436135\\
49	0.357012910714907\\
50	0.364225049918201\\
51	0.37158288399786\\
52	0.379089356185624\\
53	0.386747469170453\\
54	0.394560286299653\\
55	0.402530932804246\\
56	0.410662597049112\\
57	0.418958531808372\\
58	0.427422055566543\\
59	0.436056553845978\\
60	0.444865480561116\\
61	0.453852359400097\\
62	0.463020785234284\\
63	0.472374425556257\\
64	0.481917021946862\\
65	0.491652391571892\\
66	0.501584428709004\\
67	0.511717106305481\\
68	0.522054477567465\\
69	0.532600677581296\\
70	0.543359924967596\\
71	0.554336523568779\\
72	0.565534864170642\\
73	0.576959426258738\\
74	0.588614779810233\\
75	0.600505587121953\\
76	0.612636604675371\\
77	0.625012685039255\\
78	0.637638778810768\\
79	0.65051993659577\\
80	0.663661311029125\\
81	0.677068158835827\\
82	0.690745842933759\\
83	0.704699834578932\\
84	0.718935715554059\\
85	0.733459180401346\\
86	0.748276038700379\\
87	0.763392217392039\\
88	0.778813763149355\\
89	0.794546844796252\\
90	0.810597755775167\\
91	0.826972916664507\\
92	0.843678877746964\\
93	0.860722321629723\\
94	0.878110065917583\\
95	0.895849065940101\\
96	0.913946417533809\\
97	0.932409359880645\\
98	0.951245278403723\\
99	0.970461707721603\\
100	0.990066334662233\\
};
\addlegendentry{exponential};

\addplot [color=mycolor2,solid,densely dotted]
  table[row sep=crcr]{%
1	0\\
2	0\\
3	0\\
4	0\\
5	0\\
6	0\\
7	0\\
8	0\\
9	0\\
10	0\\
11	0\\
12	0\\
13	0\\
14	0\\
15	0\\
16	0\\
17	0\\
18	0\\
19	0\\
20	0\\
21	0\\
22	0\\
23	0\\
24	0\\
25	0\\
26	0\\
27	0\\
28	0\\
29	0\\
30	0\\
31	0\\
32	0\\
33	0\\
34	0\\
35	0\\
36	0\\
37	0\\
38	0\\
39	0\\
40	0\\
41	0\\
42	0\\
43	0\\
44	0\\
45	0\\
46	0\\
47	0\\
48	0\\
49	0\\
50	0\\
};
\addlegendentry{deterministic};

\addplot [color=red,densely dotted,forget plot]
  table[row sep=crcr]{%
51	1\\
52	1\\
53	1\\
54	1\\
55	1\\
56	1\\
57	1\\
58	1\\
59	1\\
60	1\\
61	1\\
62	1\\
63	1\\
64	1\\
65	1\\
66	1\\
67	1\\
68	1\\
69	1\\
70	1\\
71	1\\
72	1\\
73	1\\
74	1\\
75	1\\
76	1\\
77	1\\
78	1\\
79	1\\
80	1\\
81	1\\
82	1\\
83	1\\
84	1\\
85	1\\
86	1\\
87	1\\
88	1\\
89	1\\
90	1\\
91	1\\
92	1\\
93	1\\
94	1\\
95	1\\
96	1\\
97	1\\
98	1\\
99	1\\
100	1\\
};
%\addlegendentry{data3};

\addplot [color=mycolor3,solid,dashed]
  table[row sep=crcr]{%
1	0.111856823266219\\
2	0.113368401959006\\
3	0.114910829196543\\
4	0.116484950144441\\
5	0.118091639111951\\
6	0.119731800766284\\
7	0.121406371406371\\
8	0.123116320299419\\
9	0.124862651083808\\
10	0.126646403242148\\
11	0.12846865364851\\
12	0.13033051819414\\
13	0.132233153496245\\
14	0.134177758694719\\
15	0.136165577342048\\
16	0.138197899391929\\
17	0.14027606329256\\
18	0.142401458190932\\
19	0.144575526254916\\
20	0.146799765120376\\
21	0.149075730471079\\
22	0.15140503875969\\
23	0.15378937007874\\
24	0.156230471191101\\
25	0.158730158730159\\
26	0.161290322580645\\
27	0.163912929451875\\
28	0.166600026656004\\
29	0.169353746104864\\
30	0.172176308539945\\
31	0.175070028011204\\
32	0.178037316621564\\
33	0.181080689555266\\
34	0.184202770409667\\
35	0.187406296851574\\
36	0.1906941266209\\
37	0.194069243906226\\
38	0.197534766118837\\
39	0.201093951093951\\
40	0.204750204750205\\
41	0.208507089241034\\
42	0.212368331634387\\
43	0.216337833160263\\
44	0.220419679068947\\
45	0.224618149146451\\
46	0.228937728937729\\
47	0.233383121732636\\
48	0.237959261374453\\
49	0.242671325956125\\
50	0.247524752475248\\
51	0.252525252525253\\
52	0.257678829107401\\
53	0.262991794656007\\
54	0.268470790378007\\
55	0.274122807017544\\
56	0.279955207166853\\
57	0.285975749256463\\
58	0.292192613370734\\
59	0.298614429049212\\
60	0.305250305250305\\
61	0.31210986267166\\
62	0.319203268641471\\
63	0.326541274817137\\
64	0.334135257952419\\
65	0.341997264021888\\
66	0.350140056022409\\
67	0.358577165806081\\
68	0.367322950337937\\
69	0.376392652815417\\
70	0.385802469135802\\
71	0.395569620253165\\
72	0.405712431028887\\
73	0.416250416250416\\
74	0.427204374572796\\
75	0.43859649122807\\
76	0.45045045045045\\
77	0.46279155868197\\
78	0.475646879756469\\
79	0.489045383411581\\
80	0.503018108651911\\
81	0.5175983436853\\
82	0.532821824381927\\
83	0.548726953467954\\
84	0.565355042966983\\
85	0.582750582750583\\
86	0.600961538461538\\
87	0.620039682539683\\
88	0.640040962621608\\
89	0.661025912215759\\
90	0.683060109289617\\
91	0.706214689265537\\
92	0.730566919929866\\
93	0.756200846944949\\
94	0.783208020050125\\
95	0.811688311688312\\
96	0.841750841750842\\
97	0.873515024458421\\
98	0.90711175616836\\
99	0.942684766214178\\
100	0.980392156862745\\
};
\addlegendentry{Pareto};

\end{axis}
\end{tikzpicture}%
\caption{Values of $\omega_i(N)$, the probability that a customer arriving in the $i$-th time slot is still in the system at time $1$, where $N=100$, $E=0.5$.}
\label{fig:omega}
\end{figure}  

In our example, this is reflected in the values of $\omega_i(N)$: bearing in mind that we fixed the value of the mean service time $E$, the arrival rates in the system with Pareto service times are thinned less in early slots but more in later slots, compared to deterministic service times (see Fig.~\ref{fig:omega}). %Thus, it is not obvious which type of service distribution performs `better' in the sense that fewer servers are needed; this is determined by the specific values of $\omega_i(N)$. 
That Pareto service times turn out to be better is a result of the fact that %we compared the different service time distributions by fixing their mean to the same level. %The probability that exponential service times with mean $E$ take a value smaller than $E$ is rather large, and even more so for the Pareto service times. Values to the right of $E$ are taken with relatively small probability, but can (with small probability) be extremely large. This, however, does not matter to us: we are only interested in whether or not the customer remains in the system, and how much longer he stays if he has not left at time $1$ is not taken into account. 
 the Pareto service times are smaller than $E$ with large probability, and hence the regime in which the Pareto servers outperform the deterministic servers matters more than the regime in which the deterministic servers are better.  
Formally, we have that the sum of $\omega_i(N)$ is smallest in the case of Pareto servers, and hence, $S_N=\sum_{i=1}^N \Pois\big(X_i\omega_i(N)\big)$ has the smallest exceedance probability in that case. %It is in accordance with this that the expected number of customers at time $1$, in the table denoted by $M_1$, is smallest in the case of Pareto distributed servers.
%\footnote{\tt \tiny J: I made that last paragraph a bit shorter.} 

 %Looking at the specific values of $\omega_i(N)$ we see that they are not symmetric --  For the example at hand we may reason as follows. Suppose we were looking at a system with a fixed number of deterministic servers. Then if more customers arrive than expected at a later slot, all of them will still be in the system at time $1$. For Pareto servers on the other hand, there is still a positive probability that some of the customers have completed their service. 

%then if the number of arrivals in a certain slot is much larger than the number of servers, the service cannot be completed for all of them if service times are deterministic, whereas the Pareto servers may still be able to complete serving more of these customers in time. Of course, the opposite is also true, a smaller number of customers may still not complete their service if servers are Pareto. However, outliers of the arrival rates to the left of $\lambda$ are less extreme, since they are truncated  at zero. Outliers to the right of $\lambda$ may instead be arbitrarily large with positive probability.

To further investigate this issue, it is instructive to compute the variance of the steady-state number of clients in the system for the three models for the infinite-server queue. To this end, we can use the formulae that were provided in \cite[Eqn.\ (2.31)]{heemskerk2016scaling} for the special case of exponential service times, noting that they can analogously be derived for more general service time distributions. We obtain
\[\Var\left(\sum_{i=1}^N Z_i\right) = \Var X\, \sum_{i=1}^N \omega_i^2(N) +{\mathbb E}X\,\sum_{i=1}^N \omega_i(N).\]
In case the service times are typically considerably smaller than 1, this
behaves as 
\begin{equation}\label{VARE}N \,\Var X\int_0^1 \overline F\,^2(x){\rm d}x + N\,{\mathbb E}X \int_0^1 \overline F(x){\rm d}x\approx N \,\Var X\int_0^\infty \overline F\,^2(x){\rm d}x + N\,{\mathbb E}X \int_0^\infty \overline F(x){\rm d}x
.\end{equation}
In this decomposition the second part can be interpreted as the variance that one would obtain if the arrival process were Poisson with a constant (non-random) rate ${\mathbb E}X$, whereas the first part is the contribution due to overdispersion.
In our example, because $X$ has a Poisson distribution, ${\mathbb E}X=\lambda=\Var X$. 

The mean number in the system in stationarity is
\begin{align}\label{meanloadstat}
M_\infty:=N\,{\mathbb E}X\, \int_0^\infty \overline F(x){\rm d}x = N\, \lambda\, E,
\end{align}
which shows that this term depends on the service-time distribution only through its mean $E$.

It thus follows that the second term in the right-hand side of (\ref{VARE}) equals $N\, \lambda\, E$. 
We now consider the first (overdispersion-related) term.  
In the exponential case,
\[\int_0^\infty \overline F\,^2(x){\rm d}x = \int_0^\infty e^{-2x/E}{\rm d}x = \frac{E}{2};\]
in the deterministic case,
\[\int_0^\infty \overline F\,^2(x){\rm d}x = \int_0^E {\rm d}x = E;\]
and in the Pareto(2) case,
\[\int_0^\infty \overline F\,^2(x){\rm d}x = \int_0^\infty (1+x/E)^{-4} {\rm d}x = \frac{E}{3}.\]
These computations confirm that the variability in the number of clients in the system is highest when the service times are deterministic, and lowest when they are Pareto(2). This entails that -- as we saw from the results in Table \ref{TaQ} -- if there is overdispersion (i.e., $\Var X>0$), the Pareto(2) case allows for a relatively conservative staffing policy, whereas in the deterministic case comparatively many servers are required. 

%Note that the required number of servers, $N\overline a$, can be substantially lower than the expected number of customers in the system \emph{in stationarity} (as evaluated in (\ref{meanloadstat}); in contrast with $M_1$, the expected number of customers at time $1$). This shows that it can be quite crucial to consider the transient behaviour of the system.

The table also shows that the required number of servers given by $N\bar a$ is, for obvious reasons, larger than $M_1$, the expected number of customers at time 1. At the same time, $N\bar a$ can be substantially {\it lower} than the expected number of customers in the system in stationarity (i.e., $M_\infty$, as defined in (\ref{meanloadstat})),
due to the fact that the system has not necessarily reached stationarity at time $t=1$ (recall that the system starts empty at time $0$).

\subsubsection{Bursty arrival rate parameters}\label{sec:AboutQ:burstyX}

In a second example we assume that the arrivals are Poisson and usually occur with a certain rate $\lambda_1$, but occasionally occur with some larger rate $\lambda_2$ (corresponding to peak times in the network). Queueing networks with such `bursty' arrival behaviour are of interest in the context of cloud computing, see for example \cite{patch2016online,zhang2016burstiness}.

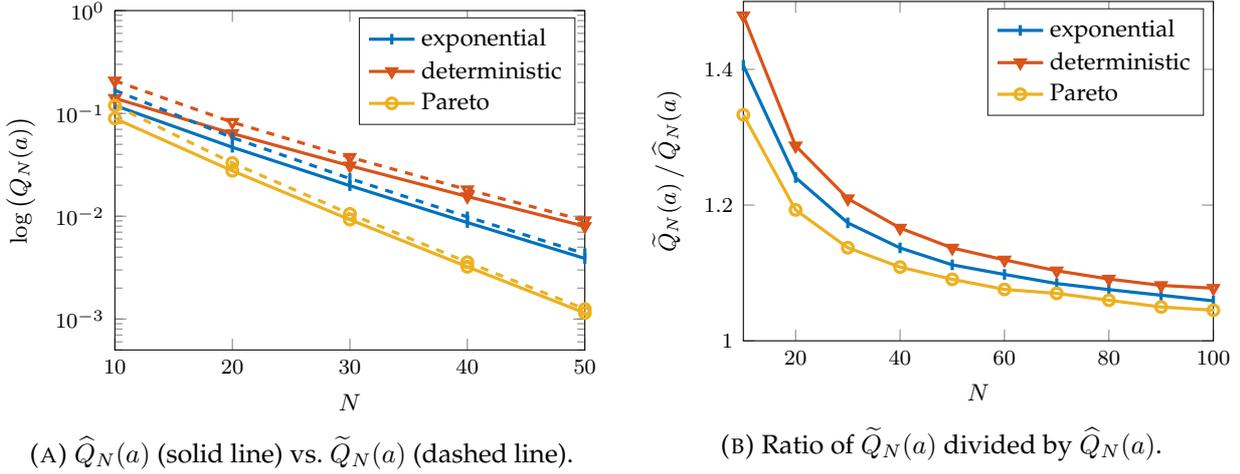
\begin{figure}
\centering
\begin{subfigure}{.49\linewidth}\centering
\definecolor{mycolor1}{rgb}{0.00000,0.44700,0.74100}%
\definecolor{mycolor2}{rgb}{0.85000,0.32500,0.09800}%
\definecolor{mycolor3}{rgb}{0.92900,0.69400,0.12500}%
\begin{tikzpicture}

\begin{axis}[%
width=0.951\figurewidth,
height=\figureheight,
at={(0\figurewidth,0\figureheight)},
scale only axis,
xmin=10,
xmax=50,
xlabel={$N$},
ymode=log,
ymin=5e-04,
ymax=1,
yminorticks=true,
ylabel={$\log\big(Q_N(a)\big)$},
axis background/.style={fill=white},
legend style={legend cell align=left,align=left,draw=white!15!black}
]
\addplot [color=mycolor1,solid,mark=|,mark options={solid}]
  table[row sep=crcr]{%
10	0.119385732\\
20	0.046848618\\
30	0.019798946\\
40	0.008669584\\
50	0.003882834\\
60	0.001758222\\
70	0.000807082\\
80	0.000372798\\
90	0.00017347\\
100	8.1222e-05\\
110	3.7622e-05\\
120	1.8264e-05\\
130	8.316e-06\\
140	4.116e-06\\
150	1.872e-06\\
};
\addlegendentry{exponential};

\addplot [color=mycolor2,solid,mark=triangle,mark options={solid,rotate=180}]
  table[row sep=crcr]{%
10	0.140288094\\
20	0.063720998\\
30	0.03096287\\
40	0.015550176\\
50	0.007977738\\
60	0.004135702\\
70	0.00217199\\
80	0.001148608\\
90	0.000610764\\
100	0.000325186\\
110	0.000175324\\
120	9.4038e-05\\
130	5.0614e-05\\
140	2.7908e-05\\
150	1.4856e-05\\
};
\addlegendentry{deterministic};

\addplot [color=mycolor3,solid,mark=o,mark options={solid}]
  table[row sep=crcr]{%
10	0.089412666\\
20	0.02764247\\
30	0.009263356\\
40	0.003219532\\
50	0.001145326\\
60	0.000414722\\
70	0.00015105\\
80	5.581e-05\\
90	2.0782e-05\\
100	7.554e-06\\
110	3.004e-06\\
120	1.178e-06\\
130	3.8e-07\\
140	1.46e-07\\
150	3.8e-08\\
};
\addlegendentry{Pareto};

\addplot [color=mycolor1,dashed,forget plot,mark=|,mark options={solid}]
  table[row sep=crcr]{%
10	0.167825624222813\\
20	0.0581176151944621\\
30	0.0232394939724966\\
40	0.00985647922155014\\
50	0.00431748975727622\\
60	0.00193021285935938\\
70	0.000875177858116835\\
80	0.00040092658959364\\
90	0.00018511990139079\\
100	8.60080055174007e-05\\
110	4.01612337518549e-05\\
120	1.88311688342245e-05\\
130	8.86055112387572e-06\\
140	4.18150774731278e-06\\
150	1.97840775538699e-06\\
};
\addplot [color=mycolor2,dashed,forget plot,mark=triangle,mark options={solid,rotate=180}]
  table[row sep=crcr]{%
10	0.207491910667676\\
20	0.0820343752176991\\
30	0.0374507007286526\\
40	0.0181342785563259\\
50	0.00906890763702378\\
60	0.00462886013683623\\
70	0.00239612994427911\\
80	0.00125321070451579\\
90	0.000660628852663997\\
100	0.000350419684994369\\
110	0.000186810746150897\\
120	0.000100003978351841\\
130	5.37212274471986e-05\\
140	2.89443158137461e-05\\
150	1.5634765768108e-05\\
};
\addplot [color=mycolor3,dashed,forget plot,mark=o,mark options={solid}]
  table[row sep=crcr]{%
10	0.119187442650273\\
20	0.032974524855721\\
30	0.0105340641682724\\
40	0.00356935409524054\\
50	0.00124910132833928\\
60	0.000446138897708349\\
70	0.000161606839822465\\
80	5.91461944973359e-05\\
90	2.18179322385659e-05\\
100	8.09837505409091e-06\\
110	3.02109641978799e-06\\
120	1.13170529669451e-06\\
130	4.25417487286127e-07\\
140	1.60393232167692e-07\\
150	6.06271942668057e-08\\
};
\end{axis}
\end{tikzpicture}%
\caption{$\hat Q_N(a)$ (solid line) vs.\ $\check Q_N(a)$ (dashed line).}
 \end{subfigure}
 \hfill
 \begin{subfigure}{.49\linewidth}\centering
\definecolor{mycolor1}{rgb}{0.00000,0.44700,0.74100}%
\definecolor{mycolor2}{rgb}{0.85000,0.32500,0.09800}%
\definecolor{mycolor3}{rgb}{0.92900,0.69400,0.12500}%
\begin{tikzpicture}

\begin{axis}[%
width=0.951\figurewidth,
height=\figureheight,
at={(0\figurewidth,0\figureheight)},
scale only axis,
xmin=10,
xmax=100,
xlabel={$N$},
ymin=1,
ymax=1.5,
ylabel={$\check Q_N(a)\,\big/\,\hat Q_N(a)$},
axis background/.style={fill=white},
legend style={legend cell align=left,align=left,draw=white!15!black}
]
\addplot [color=mycolor1,solid,mark=|,mark options={solid}]
  table[row sep=crcr]{%
10	1.40574272495823\\
20	1.24054065361036\\
30	1.17377429952567\\
40	1.13690336486158\\
50	1.11194291522023\\
60	1.09782090052301\\
70	1.08437291144746\\
80	1.07545263009362\\
90	1.06715801804802\\
100	1.05892498974909\\
110	1.06749332177595\\
120	1.03105392215421\\
130	1.06548233812839\\
140	1.01591539050359\\
150	1.05684174967254\\
};
\addlegendentry{exponential};

\addplot [color=mycolor2,solid,mark=triangle,mark options={solid,rotate=180}]
  table[row sep=crcr]{%
10	1.47904148350376\\
20	1.28739940980992\\
30	1.20953583206765\\
40	1.16617834784159\\
50	1.13677682032473\\
60	1.11924411788766\\
70	1.10319566125033\\
80	1.09106910670637\\
90	1.08164340508608\\
100	1.07759769791556\\
110	1.06551724892711\\
120	1.0634422079568\\
130	1.06139067149798\\
140	1.03713328843866\\
150	1.05242095908105\\
};
\addlegendentry{deterministic};

\addplot [color=mycolor3,solid,mark=o,mark options={solid}]
  table[row sep=crcr]{%
10	1.33300401366259\\
20	1.19289357484049\\
30	1.13717578901992\\
40	1.1086561945154\\
50	1.0906076770625\\
60	1.07575411410137\\
70	1.06988970422023\\
80	1.05977771899903\\
90	1.04984757186825\\
100	1.04506447631598\\
110	1.00569121830492\\
120	0.960700591421485\\
130	1.11951970338454\\
140	1.09858378197049\\
150	1.09545248070541\\
};
\addlegendentry{Pareto};

\end{axis}
\end{tikzpicture}%
\caption{Ratio of $\check Q_N(a)$ divided by $\hat Q_N(a)$.}
 \end{subfigure}
\caption{\small Comparison of crude Monte Carlo estimators $\hat Q_N(a)$  and the approximation $\check Q_N(a)$ as provided in Prop.\ \ref{thm:6.1}. Parameters are chosen as $E=0.5$, $p=0.75$, $\lambda_1=1$ and $\lambda_2=5$, with $a=1.6$ for deterministic, $a=1.4$ for exponential and $a=1.2$ for Pareto service times.}
\label{fig:ratio2}
\end{figure}

\begin{table}[h!]
\vspace{1cm}
%\caption{\label{TaQ2}\small Values of $a(\ve)$ needed to achieve $|\check Q_N\big(a(\ve)\big)- \ve|<10^{-9}$ with $N=100$, mean arrival rate $2$ (with arrival rate parameters $p=0.75$, $\lambda_1=1$ and $\lambda_2=5$), and mean service time $E$.  The Monte Carlo estimates $\hat Q_N\big(a(\ve)\big)$ are also provided (based on $10^9$ runs) together with $\mbox{\sc ci}$, the width of the standard normal 95\% confidence interval, as well as the values of the approximation $\check Q_N(a)$ with $\underline{\smash{ a}}$ ($\overline a$, respectively) such that $N \underline{\smash{a}}=\lfloor N a\rfloor$ ($N \overline a=\lceil N a\rceil$, respectively). The inferred number of servers is $N\overline a$, which should be larger than the expected number of customers $M_1$ at time $1$.}
\caption{\label{TaQ2}\small Parameters are chosen as in Table \ref{TaQ}, with arrival rate parameters $p=0.75$, $\lambda_1=1$ and $\lambda_2=5$ (so that the expected arrival rate is $2$).}
\begin{tabular}{>{\small}C{0.8cm}>{\small}C{0.8cm}>{\small}C{0.6cm}>{\small}C{1.1cm}>{\small}C{0.6cm}>{\small}C{0.6cm}>{\small}C{3.2cm}>{\small}C{3cm}}
\midrule
$F$ & $\ve$ & $E$  & $a\left(\ve\right)$ & $N\overline a$ & $\lceil M_1\rceil$ & $\left[\hat Q_N\left(a\left(\ve\right)\right)\pm \frac{\mbox{\sc ci}}{2}\right]\big/\ve$ & $\left(\check Q_N(\underline a),\check Q_N(\overline a)\right)\big/\ve$\\
\midrule
\multirow{6}{*}{\rotatebox[origin=c]{90}{\parbox[c]{2.3cm}{\centering Exponential}}} 
& \multirow{3}{*}{$10^{-3}$} 
& $0.05$ & $0.2662$ & $27$ & $10$ &$0.7501\pm 0.0017$ & $(1.4061,0.8115)$ \\
&& $0.5$ & $1.2991$ & $130$ & $87$ & $0.9002\pm 0.0019$ & $(1.2266,0.9787)$\\
&& $1$ & $1.8061$ & $181$ & $127$ & $0.8576\pm 0.0018$ & $(1.1182,0.9307)$\\
\cmidrule{2-8}
& \multirow{3}{*}{$10^{-4}$} 
& $0.05$ & $0.3056$ & $31$ & $10$ & $0.7199\pm 0.0053$ & $(1.4107,0.7615)$\\
&& $0.5$ & $1.3942$ & $140$ & $87$ & $0.8089\pm 0.0056$ & $(1.1124,0.8601)$  \\
&& $1$&  $1.9234$ & $193$ & $127$ & $0.8230\pm 0.0056$ & $(1.0742,0.8717)$\\
\midrule
\multirow{6}{*}{\rotatebox[origin=c]{90}{\parbox[c]{2.3cm}{\centering Deterministic}}}
& \multirow{3}{*}{$10^{-3}$} 
& $0.05$ & $0.3012$ & $31$ & $10$ & $0.6173\pm 0.0015$ & $(1.0539,0.6640)$ \\
&& $0.5$ &$1.5438$ & $155$ & $100$ & $0.8215\pm 0.0018 $ & $(1.0708,0.8934)$ \\
&& $1$ & $2.7487$ & $275$ & $200$ & $0.9035\pm 0.0019$ & $(1.1232,0.9827)$ \\
\cmidrule{2-8}
& \multirow{3}{*}{$10^{-4}$}
& $0.05$ & $0.3484$ & $35$ & $10$ & $0.8783\pm 0.0058$ & $(1.5388,0.9209)$\\
&& $0.5$ & $1.6632$ & $167$ & $100$ & $0.8187\pm 0.0056$ & $(1.0669,0.8690)$\\
&& $1$ & $2.9094$ & $291$ & $200$ & $0.9316\pm 0.0060$ & $(1.1532,0.9905)$\\
\midrule
\multirow{6}{*}{\rotatebox[origin=c]{90}{\parbox[c]{2.3cm}{\centering Pareto(2)}}} 
& \multirow{3}{*}{$10^{-3}$} 
& $0.05$ & $0.2461$ & $25$ & $10$ & $0.7264\pm 0.0017$ & $(1.4490,0.7888)$\\
&& $0.5$ & $1.0381$ & $104$ & $67$ & $0.8755\pm 0.0018$ & $(1.2856,0.7069)$\\
&& $1$ & $1.4671$ & $147$ & $100$ & $0.8651\pm 0.0018$ & $(1.1606,0.9393)$\\
\cmidrule{2-8}
& \multirow{3}{*}{$10^{-4}$}
& $0.05$ & $0.2817$ & $29$ & $10$ & $0.5315\pm 0.0045$ & $(1.1255,0.5649)$\\
&& $0.5$ & $1.1200$ & $113$ & $67$ & $0.6948\pm 0.0052$ & $(1.0002,0.7408)$\\
&& $1$ & $1.5688$ & $157$ & $100$ & $0.9138\pm 0.0059$ & $(1.2335,0.9709)$\\
\midrule
\end{tabular}
\end{table}

Specifically, we assume that $\pp(X_i=\lambda_1)=p$ and $\pp(X_i=\lambda_2)=1-p=:\overline p$, where $p$ is typically substantially larger than $\frac{1}{2}$. A routine calculation shows that
\[
\Lambda_X(\vartheta)=\log\left(p e^{\vartheta\lambda_1}+\bar pe^{\vartheta\lambda_2}\right),\:\:
\Lambda_X'(\vartheta)=\frac{\lambda_1p e^{\vartheta\lambda_1}+\lambda_2\bar pe^{\vartheta\lambda_2}}{p e^{\vartheta\lambda_1}+\bar pe^{\vartheta\lambda_2}},\:\:
\Lambda_X''(\vartheta)=\frac{p\bar p (\lambda_1-\lambda_2)^2 e^{\vartheta(\lambda_1+\lambda_2)}}{\left(p e^{\vartheta\lambda_1}+\bar pe^{\vartheta\lambda_2}\right)^2}.
\]
As before, we evaluate the approximation provided in Prop.~\ref{thm:6.1} numerically. The obtained approximations and the corresponding Monte Carlo estimates are depicted in Fig.~\ref{fig:ratio2}. The counterpart to Table \ref{TaQ} is Table \ref{TaQ2}, where the  parameters are chosen as in Section \ref{sec:AboutQ:PoisX} (we put $\lambda_1=1$, $\lambda_2=5$ and $p=0.75$ so that the mean arrival rate is $2$  as before). Compared to the previous example, it seems that here the required number of servers is overall somewhat larger  due to the greater variance of the $X_i$. The ordering of the service time distributions in terms of the required number of servers remains the same as before: the queuing system with deterministic service times requires the largest number of servers. %The difference is, however, very small,  suggesting that the distribution of the arrival rates matters little compared to the service time distribution.\footnote{\tt\tiny J: I changed this paragraph -- too strong?}

\section{Conclusion}

In this paper we considered an infinite-server queue with doubly stochastic Poisson arrivals, where the arrival rate is resampled every $N^{-\alpha}$ time units. 
Among the main contributions of the paper are exact (non-logarithmic, that is) asymptotic expressions for $P_N(a)$, namely the tail distribution of the number of arrivals at a given time (for $\alpha>3$ or $\alpha<\frac 1 3$), as well as for $Q_N(a)$, for which we consider the tail probability of having more than $N a$ customers in the system (for the case $\alpha=1$). 

As we saw for the specific example of exponentially distributed arrival rates, the asymptotic expression for $P_N(a)$ can have a rather intricate shape for $\alpha\in\left[\frac 1 2,2\right]$. We do, however, believe that it is possible to derive the asymptotics for the cases $\alpha\in\left[\frac 1 3,\frac 1 2\right)$ and $\alpha\in(2,3]$ by using more precise bounds based on the Berry-Esseen inequality. 

In numerical examples we showed how the approximation for $Q_N(a)$ can be useful when determining the required number of servers such that at a specific time $t$ (e.g.\ a certain time of the day) a specific performance target is met.  
This staffing rule could be extended to one that achieves the desired performance level during an extended period of time, rather than at a single time point. 
We expect that this requires more refined techniques, since the staffing level at a certain point in time affects the number of customers present in the subsequent time interval.
However, we feel that the procedure developed in this paper may serve as a reasonably accurate proxy.

Finally, we believe that it is possible to extend the results of the paper by relaxing the assumption that the arrival rates are independent and identically distributed. 
Instead, one could consider the situation in which the arrival rates in subsequent time intervals depend on each other in a Markovian fashion.
Another interesting topic relates to the infinite-server model in which the random rate of the arrival process changes continuously (rather than being redrawn periodically, and then being valid for the rest of the interval); in this context we could for instance consider a Coxian arrival process with a shot-noise rate \cite{KMB17}.

{\small
\section*{Acknowledgments and affiliations} The authors are with Korteweg-de Vries Institute for Mathematics,
University of Amsterdam, Science Park 904, 1098 XH Amsterdam, the Netherlands. J.\ Kuhn is also with
The University of Queensland,
St Lucia, Queensland, Australia, and is supported by Australian Research Council (ARC) grant DP130100156.
M.\ Mandjes is also with  {\sc Eurandom}, Eindhoven University of Technology, Eindhoven, the Netherlands,
and Amsterdam Business School, Faculty of Economics and Business, University of Amsterdam,
Amsterdam, the Netherlands. The research of M.\ Heemskerk and M.\ Mandjes is partly funded by NWO Gravitation project {\sc Networks}, grant number 024.002.003. 

\noindent
{\sc Email.} {\tiny \{\,{\tt j.m.a.heemskerk|j.kuhn|m.r.h.mandjes}\,\}\,{\tt @uva.nl}}.}

\bibliographystyle{plain}
\bibliography{biblioExAsPois}

\end{document}